\documentclass{svjour3}

\smartqed

\journalname{}
\usepackage{amsmath,latexsym}
\usepackage{amssymb}

\usepackage{amsfonts}
\usepackage[nesting]{hyperref}
\usepackage{color}

\newtheorem{defn}[theorem]{Definition}

\newtheorem{lem}[theorem]{Lemma}
\newtheorem{prop}[theorem]{Proposition}
\newtheorem{cor}[theorem]{Corollary}

\usepackage{pgfplots}

\usepackage{multicol}
\usepackage{multirow}
\usepackage{array}
\usepackage{float}
\usepackage{caption}
\usepackage{comment}
\usepackage{stackrel}
\usepackage{subfig}
\usepackage{algorithm}
\usepackage{algorithmicx}
\usepackage[noend]{algpseudocode}
\algdef{SE}{Begin}{End}{\textbf{begin}}{\textbf{end}}

\allowdisplaybreaks

\usepackage{alphalph}
\usepackage{etoolbox}

\newcommand\ML[1]{{\color{black} #1}}
\usepackage[justo]{optional} 
\newcommand{\justo}[1]{{\color{black}{\rm #1}}}
\newcommand\MR[1]{{\color{black} #1}}
\newcommand\MRM[1]{{\color{black} #1}}

\usepackage{xcolor,ulem}

\AtBeginDocument{%
  \AtBeginEnvironment{subequations}{%
  }
}

\newcommand{\dsum}{\displaystyle\sum}

\newcommand{\dmin}{\displaystyle\min}

\def\R{\mathbb{R}}
\def\N{\mathbb{N}}

\def\sign{\mathrm{sign}}

\begin{document}
\title{Shortest paths and location problems in a continuous framework with different $\ell_p$-norms on different regions}
\titlerunning{ }
\author{Martine Labb\'e \& Justo Puerto \& Mois\'es Rodr\'iguez-Madrena}

\institute{Martine Labb\'e \at
Computer Science Department, Universit\'e Libre de Bruxelles, Belgium\\
INOCS team, INRIA, Lille, France\\
\email{martine.labbe@ulb.be}
\and
J. Puerto \and M. Rodr\'iguez-Madrena \at
IMUS, Instituto de Matem\'aticas, Universidad de Sevilla, Spain\\
Departamento de Estad\'istica e Investigaci\'on Operativa, Universidad de Sevilla, Spain\\
\email{puerto@us.es}
\email{madrena@us.es}}

\date{\today}
\maketitle

\begin{abstract}
\MRM{In this paper we address two different related problems. We first study the problem of finding a simple shortest path \ML{in} a $d$-dimensional real space subdivided in several polyhedra endowed with different $\ell_p$-norms. This problem is a variant of the weighted region problem, a classical path problem in computational geometry introduced in Mitchell and Papadimitriou (JACM 38(1):18-73, 1991). As done in the literature for other geodesic path problems, we relate its local optimality condition with Snell's law and provide an extension of this law in our framework space. We propose a solution scheme based on the representation of the problem as a mixed-integer second order cone problem (MISOCP) using the $\ell_p$-norm modeling procedure given in Blanco et al. (Comput Optim Appl 58(3):563–595, 2014). We derive two different MISOCPs formulations, theoretically compare the lower bounds provided by their continuous relaxations, and propose a preprocessing scheme to improve their performance. The usefulness of this approach is validated \ML{through}  computational experiments.  \ML{The formulations provided are flexible since some extensions of the problem can be \justo{handled} by  transforming the input data in a simple way}. The second problem that we consider is the Weber problem that results in this subdivision of $\ell_p$-normed polyhedra. To solve it, we adapt the solution scheme that we developed for the shortest path problem and validate our methodology with extensive computational experiments.}

\keywords{ Geodesic distances, weighted region problem, location with different norms.}
\subclass{.}
\end{abstract}

\section{Introduction}
\label{s:1}

\normalem

The weighted region problem (WRP) is a generalization of the shortest path problem considered in a geometric domain where the travel distance is region-dependent. More precisely, given a \justo{subdivision} of the plane in polyhedra with different associated weights, the WRP asks for the \justo{Euclidean} shortest path between two points but taking into account that the distance traversed along a polyhedron has to be multiplied by its associated weight. The WRP was originally introduced in \cite{mitchellWRP} and, besides its mathematical interest, it was motivated as a model to
\justo{design} the route of robots through zones with different terrains that \MR{are} \justo{traversed} at different speeds (e.g. grassland, blacktop, water, et cetera). Other practical applications of the WRP have been proposed \MR{for instance} in geographical information systems (GIS) and in Seismology (see \cite{pathrefinementWRP}).


In \cite{mitchellWRP}, the authors proposed an \MR{approximation algorithm} for the WRP that is inspired in the continuous Dijkstra method and which also takes advantage of the \MR{particular} local optimality condition \justo{of} the problem. Later, most of the algorithmic approaches that have been proposed are based on the same \justo{intuitive} approximation scheme: the discretization of the intermediate faces of the polyhedra by placing Steiner points in a clever way, and replacing therefore the original geometric shortest path problem by a shortest path problem in a graph (see e.g. \cite{AMS2005,LMS2001,matamitchell,SR2006}). \MR{For further information about these approximation algorithms for the WRP and the details on their complexities, the reader is referred to the sections dedicated to the problem in \cite{mitchell2000} and \cite{HandbookDCG2017}, as well as to the illustrative comparison tables in \cite{AMS2005}.} \MR{Although these algorithms are able to provide $(1+\varepsilon)$-approximate weigthed Euclidean shortest paths, still no FPTAS for the WRP is known.} \justo{Furthermore, in} \cite{unsolvabilityWRP} it is proven that WRP is not solvable in any algebraic computation model over the rational numbers, which means that, in general, the exact solution of WRP \ML{cannot} be computed in $\mathbb{Q}$ using a finite number of the operations $+$, $-$, $\times$, $\div$, $\sqrt[k]{\textcolor{white}{x}}$, for any $k\geq 2$. As it is said in \cite{pathrefinementWRP}, the fact that WRP is not solvable in any algebraic computation model over the rational numbers seems to indicate that it is unlikely that the problem can be solved in polynomial time: the exact complexity of WRP \ML{is still unknown}. \MR{These considerations seems to explain why} previous research \justo{on this topic has only focused} on determining approximate algorithms to solve the problem.

\MRM{Among the variants of the WRP considered in the literature, we can \justo{distinguish} between the ones that affect the geometric domain where the problem is defined and those that require a specific structure \ML{for} the generated path. An example of variant belonging to the first group is the extension of the WRP \ML{considered} in \cite{CNVW2008}, where the regions of the planar subdivision are endowed with general convex, possibly asymmetric, distance functions. Also regarding the geometric domain where the problem is defined, see \cite{FortSellares2012} and the comparison tables and references in \cite{AMS2005} for a variant of the WRP in non-convex polyhedral surfaces. On the other hand, the WRP with constraints on the number of links studied in \cite{DMNPY2008} \justo{ is a variant that requires a specific structure \ML{for} the generated path. Another} variant in this group is the one in \cite{pathrefinementWRP}, where some additional qualitative criteria are required on a path for the WRP. These qualitative criteria, expressed as geometric rules that the path must obey, are based on the discussions of the authors with Seismologists and Geophysicists, and can be related with the criteria of smoothness and sharp turns avoidance previously considered in the modeling of paths in other real applications \cite{CheeTomizuka1994,Nelson1989,NieKamMooOver2004}. In all these problem variations, the proposed solution schemes are based on the above mentioned discretization by means of Steiner points.}

\MR{In this paper we consider \justo{the Weighted Region Problem with different $\ell_p$-norms ($\ell_p$-WRP)}, a variant of the WRP where each polyhedron of the subdivision is endowed with a different $\ell_{p}$-norm, $1\leq p\leq +\infty$. \MRM{
We study this problem not only in the plane, but also in $\mathbb{R}^d$. As an intuitive qualitative criterion, we restrict feasible paths to be simple, i.e., a path can visit each polyhedron at most once. In general, simple paths fit the criteria of smoothness and sharp turns avoidance better than the paths that enter to and leave the same polyhedron several times.} Since the WRP is not solvable in any algebraic computation model over the rational numbers, \justo{the $\ell_p$-WRP in this paper is not either.} 

We propose a solution scheme for the problem completely different \ML{from} the ones proposed for the WRP and its \ML{variants previously considered. This  scheme} consists in the representation of the problem as a mixed-integer second order cone problem (MISOCP), which is achieved using the $\ell_p$-norms modeling procedure given in \cite{BPE2014,refraction2017}.} \MRM{MISOCPs} can be solved nowadays up to any degree of accuracy with commercial solvers. Therefore, we propose an approximate method with a high precision degree whose accuracy \MR{is only compromised by} the rounding error \MR{inherent to computers} and the interior point algorithm in use. \MRM{In particular we present two MISOCP formulations for the problem, theoretically compare their properties, and propose a preprocessing scheme to improve their performance. The performance of the formulations and the preprocessing are evaluated on a testbed of instances. We also discuss how the developed MISOCP formulations can be adapted \ML{to deal} with two extensions of the problem: the case when rapid transit boundaries are considered and the case of non-simple paths. These extensions can be handled by  applying a \ML{simple} transformation \ML{to} the input data, which shows the flexibility of our approach.}

\MR{As commented above, in the seminal paper  \cite{mitchellWRP}, the authors propose an approximation algorithm for the WRP that exploits the local optimality condition of  the problem. This condition \justo{establishes that a weighted Euclidean} shortest path ``obeys'' Snell's law at the \ML{intersection} face of two polyhedra (see Proposition 3.5 (b) in \cite{mitchellWRP}). 
Snell's law is the  \MRM{physical law} that governs the phenomenon of refraction for light and other waves (see Section \ref{ss:22} for a formal statement). As pointed out in \cite{pathrefinementWRP}, among all the qualitative criteria for a weighted Euclidean shortest path, Snell's law is the most prominent one.} The relation between some classes of geodesic paths and Snell's law has \MR{also} been pointed out in \cite{refraction2017,lyusternik1964,matamitchell,mitchellWRP,Warntz1957} among others. \MR{Indeed,} Snell's law is derived from Fermat's principle, which states that a ray light between two points follows the path that provides the least time.
It is this ``shortest path behaviour'' that motivates the relation between geodesics and Snell's law. \MR{In line with the attention paid to this relation, and as a by-product of our work, we show the extension of Snell's law that results from the \MRM{geodesic distance induced by the $\ell_p$-WRP} considered in this paper. An extension of Snell's law \MRM{in a $\mathbb{R}^d$ space endowed with different $\ell_p$-norms}  has been previously stated using \emph{generalized \ML{sines}} in \cite{refraction2017}.  
In the present paper the statement of the generalized Snell's law is based on the polarity correspondence between $\ell_p$-norms (see e.g. \cite{rockafellarCA}).
Other extensions of Snell's law in different frameworks have been presented \MRM{for instance} in \cite{GhandehariGolomb2000,Slawinski2000}.}

\MR{In \cite{FortSellares2012}, besides  presenting an approximation algorithm for the WRP in non-convex polyhedral surfaces, the authors take advantage of  that algorithm for approximately solving facility location problems in the induced framework, i.e. considering the \MRM{distance measure induced by the WRP}. Facility location problems consist in placing one or more facilities with respect to a set of given demand points in such a way that some \justo{objective function, often related with a distance measure between the new facilities and the demand points, is optimized.}
For a comprehensive overview of this well-known class of problems in Operation Research, the reader is referred to 
\cite{DH02} or \cite{NP05}. \justo{In particular,  the classical Weber location problem consists in minimizing} the sum of the distances between a new facility and the demand points.
Generalizations of the Weber problem in which distances between points are measured in two different ways depending on the regions where the points are located, have been \MRM{studied} in \cite{refraction2017,brimberg2003,brimberg2005,fathaly,FranVelGon2012,Parlar1994}.
For a rigorous analysis of the induced distance measures in this framework see the works \cite{plastria2019I,plastria2019II}. In \cite{refraction2017},  the Weber problem \ML{is considered in a} real space $\mathbb{R}^d$ endowed with two different $\ell_p$-norms. That problem can be seen as \ML{a} particular case of the Weber problem \justo{in the framework induced by the $\ell_p$-WRP  of the present paper} \MRM{taking a subdivision of $\mathbb{R}^d$ consisting of} two half-spaces.
Motivated \justo{by the above comments}, we also consider in this paper such a Weber problem for a general subdivision. We use the results we obtained for \justo{the $\ell_p$-WRP}  to derive a MISOCP representation for the \MRM{Weber problem that it induces}, and therefore following a similar approach \justo{to that using SOCPs in \cite{refraction2017}. The usefulness of this approach is validated reporting extensive computational experiments.}}

\MRM{The paper is organized as follows. Section \ref{s:2} is devoted to \justo{the $\ell_p$-WRP} considered in this work. We 
formally state the problem and introduce the notation and definitions. In Section \ref{ss:21} we relate the local optimality condition for the problem with an extension of Snell's law that we derive. We propose two MISOCP formulations for the problem, study their properties, and design a preprocessing strategy in Section \ref{ss:22}. Computational experiments using these solution schemes are reported in Section \ref{ss:23}. In Section \ref{ss:24} we elaborate on different extensions of the problem. Section \ref{s:3} is devoted to the Weber problem that results in the subdivision of $\ell_p$-normed polyhedra. A solution scheme for this problem based on the one developed in Section \ref{ss:22} is presented in Section \ref{ss:31}. Extensive computational experiments using this approach are shown in Section \ref{ss:32}. Finally, in Section \ref{s:4} some \justo{concluding}  remarks are outlined.}

\section{Shortest paths between points in a continuous framework with different norms}
\label{s:2}

Let $\{P_1,...,P_m\}$ be a \justo{subdivision} in polyhedra of $\mathbb{R}^d$, i.e., $P_1,...,P_m$ are polyhedra, $\bigcup_{i=1}^m  P_i = \mathbb{R}^d$ and $\text{int}(P_i) \cap \text{int}(P_j) = \emptyset$ for all $i,j=1,...,m$, $i\not=j$.  

We consider the problem of finding \ML{a shortest simple path} between two points $x_s,x_t\in \mathbb{R}^d$ assuming that each polyhedron $P_i$ is endowed with a different $\ell_{p_i}$-norm and has associated a weight $\omega_i>0$, $i=1,...,m$. The adjacency relationship between the polyhedra in the \justo{subdivision} 
is determined by the \ML{undirected} graph \MR{$G=(V,E)$} such that $V=\{1,...,m\}$ and \MR{$\{i,j\}\in E$} iff $P_i$ and $P_j$ have a common face, $i,j=1,...,m$, $i\not=j$. We denote by $F_{ij}=P_i \cap P_j$ the \ML{maximal} common face of two polyhedra $P_i$ and $P_j$.

Assume that $s$ is the index in $\{1,...,m\}$ such that $x_s\in P_s$, and that $t$ is the index such that $x_t\in P_t$, $t\in \{1,...,m\}$. In order to formally define the problem, we introduce the set $\Gamma$ of all the simple paths from $s$ to $t$ in $G$. Then, $(s,i_1,i_2,\cdots ,i_{k-1},i_k, t)\in \Gamma$ is the vertex sequence of a simple path from $s$ to $t$ in $G$, \ML{where all vertices of the sequence are different}. 

We state and analyze the problem for general $\ell_{p}$-norms with $1\leq p_i \leq +\infty$, although we anticipate that our \justo{solution} scheme will be restricted to $\ell_{p}$-norms such that either $p=q/r$ with $q, r \in \N\setminus \{0\}$ and $\gcd(q,r)=1$, or $p=+\infty$. In the following, $\|x\|_p$ refers to the $\ell_p$-norm of $x\in \mathbb{R}^d$. 

The problem can be stated as follows. We look for a shortest path between the points $x_s,x_t\in \mathbb{R}^d$, $s\not= t$, whose length is given by the distance measure 
\begin{align}
D(x_s,x_t) = \inf_{\scriptsize\begin{array}{c}(s,i_1,i_2,\cdots ,i_{k-1},i_k, t)\in \Gamma  \\ y_{s i_1}\in F_{s i_1},\ldots,y_{i_kt}\in F_{i_kt}\end{array}} \omega_s\|y_{s i_1}-x_s\|_{p_{s}} + \omega_{i_1}\|y_{i_1 i_2}-y_{si_1}\|_{p_{i_1}} + \ldots + \nonumber\\
 \omega_{ i_k}\|y_{i_k i_t}-y_{i_{k-1}i_k}\|_{p_{i_k}} + \omega_{t}\|x_t-y_{i_k t}\|_{p_{t}}. \label{SPP} \tag{${\rm SPP}$}
\end{align}
The $y$ points \ML{belonging to} the faces of the polyhedra are called \emph{gate points}. If $s= t$, then we consider $D(x_s,x_t) = \|x_t-x_s\|_{p_{s}}$. In case one of the points $x_s$ or $x_t$ is \ML{on a face that belongs to} two or more polyhedra, we assume it belongs to exactly one of them which is \ML{determined} in advance. An illustrative instance of Problem \eqref{SPP} is shown in Fig. \ref{instanceSPP}.

\begin{figure}[H]
\centering
\includegraphics[scale=0.6]{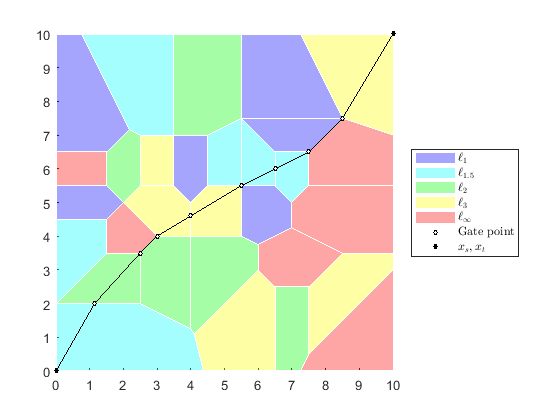}
\caption{An instance of Problem \eqref{SPP}. }\label{instanceSPP}
\end{figure}

It is not difficult to see that Problem (\ref{SPP}) is well-defined given that for any simple path of the finite set $\Gamma$, the location of the gate points can be bounded, so the problem reduces to determine some points in some compact sets in order to minimize a continuous function. Then, in the following w.l.o.g. we assume that $P_1,...,P_m$ are polytopes (i.e., bounded polyhedra) and therefore their faces $F_{ij}$ are also polytopes.  

\ML{Each polytope $P_i$ is encoded by its set $\text{Ext}(P_i)$  of extreme points. Any point $x$ in a polytope $P_i$ can be represented as the convex combination of its extreme points: $x=\sum_{e\in \text{Ext}(P_i)} \lambda_e e$ for some scalars $\lambda_e\geq 0$, $e\in \text{Ext}(P_i)$, with $\sum_{e\in \text{Ext}(P_i)} \lambda_e= 1$. Further, for  a face  $F_{ij}$, $\text{Ext}(F_{ij})=\text{Ext}(P_i) \cap \text{Ext}(P_j)$.}

For any \ML{optimal} solution $(y_{s i_1},y_{i_1i_2},\ldots,y_{k-1k},y_{kt})\in F_{s i_1}\times F_{i_1i_2}\times \ldots \times F_{k-1k}\times F_{kt}$ of Problem (\ref{SPP}) for a fixed $\gamma=(s,i_1,i_2,\cdots ,i_{k-1},i_k, t)\in \Gamma$, we call the piecewise linear path $[x_s,y_{s i_1}]\cup [y_{s i_1},y_{i_1i_2}]\cup \ldots \cup [y_{k-1k},y_{kt}]\cup [y_{kt},x_t]\subseteq \mathbb{R}^2$ a \emph{geodesic path} between the points $x_s$ and $x_t$, and we refer to $\{x_s,y_{s i_1},y_{i_1i_2},\ldots,y_{k-1k},y_{kt},x_t\}$ as the \emph{breaking points} of the geodesic path. A shortest path is therefore a geodesic path with minimum length.

In Problem (\ref{SPP}), a shortest path is associated to a path $\gamma = (s,i_1,i_2,\cdots ,i_{k-1},i_k, t) \in \Gamma$ in the graph $G$. In our approach, such a path $\gamma$ is by definition simple, which means that \ML{each vertex appears at most once in $\gamma$}. Note that this approach implies that in Problem (\ref{SPP}) a shortest path can visit each polyhedron \ML{at most once.}

The distance measure $D$ induced by Problem (\ref{SPP}) does not define a metric because it does not verify the triangle inequality. To clarify this statement consider the situation illustrated in Fig. \ref{Counterexample}. \ML{The polyhedra $P_1 = \{x\in\mathbb{R}^2: (-1,1) x \geq 5\}$, $P_2 = \{x\in\mathbb{R}^2: 0 \leq (-1,1) x \leq 5\}$ and $P_3 = \{x\in\mathbb{R}^2: (-1,1) x \leq 0\}$ constitute a \justo{subdivision} of $\mathbb{R}^2$.} Polyhedron $P_i$ is endowed with the $\ell_1$-norm and \ML{its associated weight is $\omega_i = i$, $i=1,2,3$.} Taking the points 
$x_s = (1,0)^T$, $x_u = (1,9)^T$ and $x_t = (10,9)^T$, one can \ML{verify} that $D(x_s,x_t) > D(x_s,x_u) +D(x_u,x_t)$: 
\begin{eqnarray*}
&& D(x_s,x_t)  =  3\|(10,9)^T -(1,0)^T \|_1 = 54,\\ 
&& D(x_s,x_u)  =  3\|(1,1)^T-(1,0)^T \|_1+2\|(1,6)^T-(1,1)^T\|_1+\|(1,9)^T-(1,6)^T\|_1 = 16,\\ 
&& D(x_u,x_t)  =  \|(4,9)^T - (1,9)^T\|_1 +2\|(9,9)^T - (4,9)^T\|_1+3\|(10,9)^T - (9,9)^T\|_1= 16.
\end{eqnarray*}

\begin{figure}[H]
\centering
\includegraphics[scale=0.6]{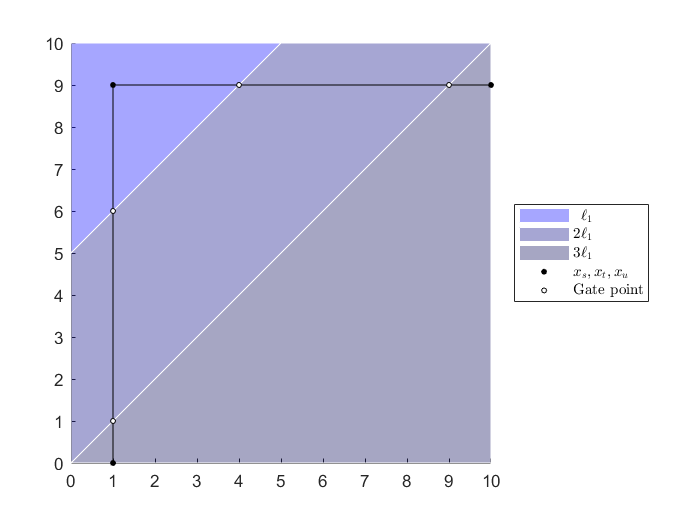}
\caption{The distance measure $D$ does not verify the triangle inequality for the points $x_s = (1,0)^T$, $x_u = (1,9)^T$, $x_t = (10,9)^T$ since $D(x_s,x_t) > D(x_s,x_u) +D(x_u,x_t)$. }\label{Counterexample}
\end{figure}

\ML{If we would allow a path to visit a polyhedron more than once, then the induced distance would verify the triangle inequality.} However, that is not the distance measure considered in this paper that focuses on the distance measure $D$ defined as in Problem (\ref{SPP}). Some extensions of the distance measure $D$ will be considered in Section \ref{ss:24}.

\MRM{Let us discuss the solvability of Problem (\ref{SPP}).} The weighted region problem (WRP) is the particular case of Problem (\ref{SPP}) in $\mathbb{R}^2$ involving weigthed Euclidean norms only, but allowing the path to visit each polyhedron more than once. In \cite{unsolvabilityWRP} it is proven that WRP is not solvable in any algebraic computation model over the rational numbers, which means that, in general, the exact solution of WRP cannot be computed in $\mathbb{Q}$ using a finite number of the operations $+$, $-$, $\times$, $\div$, $\sqrt[k]{\textcolor{white}{x}}$, for any $k\geq 2$. The result is obtained proving the unsolvability for the instance shown in Fig. \ref{Carufel} in which the following subdivision of $\mathbb{R}^2$ is considered: $P_1 = \{x\in\mathbb{R}^2: (1,0) x \leq 1\}$, $P_2 = \{x\in\mathbb{R}^2: 1\leq (1,0) x \leq 3\}$ and $P_3 = \{x\in\mathbb{R}^2: (1,0) x \geq 3\}$. Polyhedron $P_i$ is endowed with the $\ell_2$-norm and has associated a weight $\omega_i = i$, $i=1,2,3$. The starting and the terminal points are $x_s = (0,0)^T$ and $x_t = (6,2)^T$, respectively. \justo{WRP and Problem (\ref{SPP}) are equivalent for this instance since the optimal solution of WRP does not visit any polyhedron more than once. Therefore, it follows that Problem (\ref{SPP}) is also not solvable either in any algebraic computation model over the rational numbers.}

\begin{figure}[H]
\centering
\includegraphics[scale=0.6]{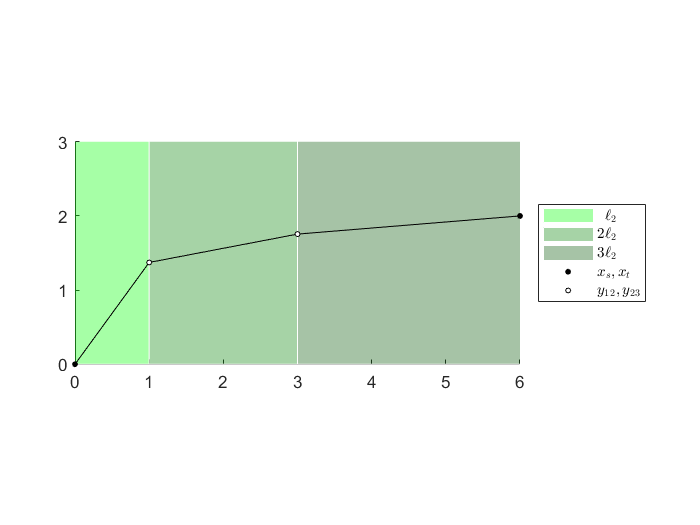}
\caption{\MRM{Unsolvable instance of the WRP from \cite{unsolvabilityWRP}.}}\label{Carufel}
\end{figure}

\subsection{Local optimality condition for gate points and its relation with Snell's law}
\label{ss:21}  

\ML{The gate points of a geodesic path are local optima} \MR{of Problem \eqref{SPP}}. In this section we provide a result that characterizes the local optimality condition for the gate points in Problem \eqref{SPP}. 

First of all, we introduce some notation. We will use the notation $\text{conv}(S)$, $\text{aff}(S)$ and  $\text{relint}(S)$  to refer to the convex hull, the affine hull and the relative interior of a set of points $S$ in $\mathbb{R}^d$, respectively. If the set of points $S$ is also an affine set, we will denote by $L(S)$ its corresponding vector subspace, i.e., $L(S) = S-S = \{x-x':x\in S,x'\in S\}$, and by $L(S)^\perp$ the orthogonal complement of $L(S)$. Finally, we will denote by $B_p$ the unit ball of the $\ell_p$-norm of \justo{$\mathbb{R}^d$}, i.e., $B_p = \{x\in\mathbb{R}^d: \|x\|_p\leq 1\}$.

The following result is a technical lemma needed to establish the local optimality condition for the gate points in Problem \eqref{SPP}.

\begin{lem} \label{techlemma}
Let $P$ be a polytope in $\mathbb{R}^d$. Consider $x\in P$ and let $\mathcal{F}_x$ be the face of $P$ such that $x\in \text{\normalfont relint}(\mathcal{F}_x)$. Then, for any $e\in \text{\normalfont Ext}(P)$, we have that $e\in \text{\normalfont Ext}(\mathcal{F}_x)$ iff there exists  \ML{a} representation of $x$ as a convex combination of the extreme points of $P$ \ML{in which} the scalar associated to $e$ is strictly positive. Moreover, there exists a representation of $x$ as \ML{a} convex combination of the extreme points of $\mathcal{F}_x$ \ML{in which} all the scalars are strictly positive.
\end{lem}
\begin{proof}
\MR{First, we note that there can only be one face $\mathcal{F}_x$ of $P$ such that $x\in \text{relint}(\mathcal{F}_x)$. By Theorem 6.9 in \cite{rockafellarCA}, we know that
$$\text{relint}(\mathcal{F}_x) =  \left\{\sum_{e\in \text{Ext}(\mathcal{F}_x)} \lambda_e e :  \lambda_{e}>0, e\in \text{Ext}(\mathcal{F}_x), \text{ and } \sum_{e\in \text{Ext}(\mathcal{F}_x)} \lambda_e= 1\right\},$$
\ML{so that} there exists a representation of $x$ as the convex combination of the extreme points of $\mathcal{F}_x$ where all the scalars are strictly positive. 

Now, let $e'\in \text{Ext}(P)\setminus \text{Ext}(\mathcal{F}_x)$ and let us assume that there exists a representation of $x$ as a convex combination of the extreme points of $P$ in which the scalar associated to $e'$ is strictly positive. Let $\mathcal{E}\subseteq \text{Ext}(P)$ be the extreme points whose associated scalars in that convex combination are strictly positive. Then, by Theorem 6.9 in \cite{rockafellarCA}, we have that $x\in \text{relint}(\text{conv}(\mathcal{E}))$. Thus, $\text{relint}(\text{conv}(\mathcal{E}))\cap \mathcal{F}_x\neq \emptyset$ and by Theorem 18.1 in \cite{rockafellarCA} it follows that $\text{conv}(\mathcal{E})\subseteq \mathcal{F}_x$. In particular that inclusion implies $e'\in \text{Ext}(\mathcal{F}_x)$, which contradicts the above \ML{statement}. \qed}
\end{proof}

\begin{prop}[Local optimality criterion for gate points] \label{propLOCCP}
Assume $1<p_i<+\infty$, $i=1,...,m$.
Let $a=(a_1,\ldots,a_d)^T$, $b=(b_1,\ldots,b_d)^T$ and $c=(c_1,\ldots,c_d)^T$ be consecutive breaking points of a geodesic path in Problem \eqref{SPP} such that $a\in P_i\setminus F_{ij}$, $b\in F_{ij}$ and  $c\in P_j\setminus F_{ij}$. Let $\mathcal{F}_b$ be the face of $F_{ij}$ such that $b\in \text{\normalfont relint}(\mathcal{F}_b)$. Then, the gate point $b$ satisfies 
\begin{equation*}
\omega_i \sum_{k=1}^d \left[ \frac{|b_k-a_k|}{\|b -a\|_{p_i}}\right]^{p_i-1} \sign(b_k-a_k)(f_k-g_k) =
\omega_j \sum_{k=1}^d \left[ \frac{|c_k-b_k|}{\|c-b\|_{p_j}}\right]^{p_j-1} \sign(c_k-b_k)(f_k-g_k),
\end{equation*}
for all points $f,g\in \mathcal{F}_b$.
\end{prop}
\begin{proof}
In Problem \eqref{SPP}, the gate points of a geodesic path  are optimally located in the common faces of the polyhedra to provide the minimum cost. In particular, the above statement implies that the gate point $b$ is an optimal solution of the problem 
\begin{equation*}
\dmin_{y\in F_{ij}} \omega_i\|y-a \|_{p_i}+\omega_j\|c-y\|_{p_j},
\end{equation*}
which can be equivalently written as
\begin{equation*}
\displaystyle \begin{array}{rcl} \dmin & & \omega_i\|y(\lambda)-a \|_{p_i}+\omega_j\|c-y(\lambda)\|_{p_j}\\
s.t. & & \sum_{e\in \text{Ext}(F_{ij})} \lambda_e = 1\\
     & & \lambda_e\geq 0\quad \forall e\in \text{Ext}(F_{ij})
\end{array} 
\end{equation*}
\ML{where} $y(\lambda) = \sum_{e\in \text{Ext}(F_{ij})} \lambda_e e$.

The last problem is a convex minimization problem with a set of linear constraints. Then, \justo{as} \ML{its associated Lagrangian function is}
\begin{equation*}
L(\lambda,\tau,\xi) = \omega_i\|y(\lambda)-a \|_{p_i}+\omega_j\|c-y(\lambda)\|_{p_j}+\tau \left(\sum\nolimits_{e\in \text{Ext}(F_{ij})} \lambda_e - 1\right)+\sum\nolimits_{e\in \text{Ext}(F_{ij})} \xi_e (-\lambda_e),
\end{equation*}
the necessary and sufficient conditions for a global optimum of the problem are: 
\begin{eqnarray*}
&&\omega_i \sum_{k=1}^d \left[ \frac{|y_k(\lambda)-a_k|}{\|y(\lambda) -a\|_{p_i}}\right]^{p_i-1} \sign(y_k(\lambda)-a_k)e_k + \omega_j \sum_{k=1}^d \left[ \frac{|c_k - y_k(\lambda)|}{\|c - y(\lambda)\|_{p_j}}\right]^{p_j-1} \sign(y_k(\lambda)-c_k)e_k \\ 
&&+ \tau - \xi_e = 0, \; \forall e\in \text{Ext}(F_{ij}),\\
\\
&&\xi_e =0, \; \forall e\in \text{Ext}(F_{ij}): \lambda_e\not= 0,\\
\\
&&\sum_{e\in \text{Ext}(F_{ij})} \lambda_e - 1 =0,\\
\\
&&\lambda_e\geq  0, \; \forall e\in \text{Ext}(F_{ij}),
\end{eqnarray*}
where $y_k(\lambda) = \sum_{e\in \text{Ext}(F_{ij})} \lambda_e e_k$.

Now, observing that there exists an optimal solution and that this is what we call the gate point b, by Lemma \ref{techlemma}, we know that there does not exist $\lambda_e \geq 0$, $e\in \text{Ext}(P)$, satisfying $b=\sum_{e\in \text{Ext}(P)} \lambda_e e$ and $\sum_{e\in \text{Ext}(P)} \lambda_e= 1$, such that $\lambda_{e'}>0$ for some $e'\in \text{Ext}(P)\setminus \text{Ext}(\mathcal{F}_b)$. Moreover, again by  Lemma \ref{techlemma}, we also know that $b=\sum_{e\in \text{Ext}(\mathcal{F}_b)} \lambda_e^b e$ for some $\lambda_e^b > 0$, $e\in \text{Ext}(\mathcal{F}_b)$, with $\sum_{e\in \text{Ext}(\mathcal{F}_b)} \lambda_e^b= 1$. Then, from the above necessary and sufficient conditions it follows that
\small \begin{equation}\label{LOCGPeq1}
\omega_i \sum_{k=1}^d \left[ \frac{|b_k-a_k|}{\|b -a\|_{p_i}}\right]^{p_i-1} \sign(b_k-a_k)e_k +
\omega_j \sum_{k=1}^d \left[ \frac{|c_k-b_k|}{\|c -b\|_{p_j}}\right]^{p_j-1} \sign(b_k-c_k)e_k = -\tau, \; \forall e\in \text{Ext}(\mathcal{F}_b).
\end{equation} \normalsize

At this point, consider $f,g\in \mathcal{F}_b$. Let $\lambda_e^{f}\geq 0$, $e\in \text{Ext}(\mathcal{F}_b)$, be any scalars such that $f=\sum_{e\in \text{Ext}(\mathcal{F}_b)} \lambda^{f}_e e$ and $\sum_{e\in \text{Ext}(\mathcal{F}_b)} \lambda^{f}_e= 1$. In the same way, let $\lambda_e^{g}\geq 0$, $e\in \text{Ext}(\mathcal{F}_b)$, satisfying $g=\sum_{e\in \text{Ext}(\mathcal{F}_b)} \lambda^{g}_e e$ and $\sum_{e\in \text{Ext}(\mathcal{F}_b)} \lambda^{g}_e= 1$.

For each $e\in \text{Ext}(\mathcal{F}_b)$, proceed as follows. 
By equating the left side of the equalities in \eqref{LOCGPeq1} associated to $e$ and any $e'\in\text{Ext}(\mathcal{F}_b)$, one can derive that
\small\begin{equation*}
\omega_i \sum_{k=1}^d \left[ \frac{|b_k-a_k|}{\|b -a\|_{p_i}}\right]^{p_i-1} \sign(b_k-a_k)(e_k-e_k') =
\omega_j \sum_{k=1}^d \left[ \frac{|c_k-b_k|}{\|c-b\|_{p_j}}\right]^{p_j-1} \sign(c_k-b_k)(e_k-e_k'), \;  \forall e'\in \text{Ext}(\mathcal{F}_b).
\end{equation*} \normalsize
Next, for each $e'\in \text{Ext}(\mathcal{F}_b)$, multiply its associated equality from the collection above by $\lambda_{e'}^g$, and then sum them all. When this process is done for all $e\in \text{Ext}(\mathcal{F}_b)$, the new set of equalities
\small\begin{equation}\label{LOCGPeq2}
\omega_i \sum_{k=1}^d \left[ \frac{|b_k-a_k|}{\|b -a\|_{p_i}}\right]^{p_i-1} \sign(b_k-a_k)(e_k-g_k) =
\omega_j \sum_{k=1}^d \left[ \frac{|c_k-b_k|}{\|c-b\|_{p_j}}\right]^{p_j-1} \sign(c_k-b_k)(e_k-g_k), \;  \forall e\in \text{Ext}(\mathcal{F}_b),
\end{equation} \normalsize
is obtained.

Finally, for each $e\in \text{Ext}(\mathcal{F}_b)$, multiply its associated equality from the collection \eqref{LOCGPeq2} by $\lambda_{e}^{f}$, and then sum them all. The necessary condition stated in the proposition results from this operation.\qed
\end{proof}

\MRM{Based on the local optimality criterion given in Proposition \ref{propLOCCP} we derive an extension of Snell's law in the context of Problem \eqref{SPP}.}

\MRM{Snell's law is the  physical law that governs the phenomenon of refraction. It states the following.} Consider two different media $\mathcal{M}_1$ and $\mathcal{M}_2$ where a planar wave have indices of refraction $n_1$ and $n_2$, respectively. \MRM{The index  of refraction is in inverse proportion to the phase velocity of the wave in the considered medium}. Depending on the context the medium can be air, water, glass, et cetera, and the typical examples of waves to illustrate the law are sound waves, water waves or the light. Then, when the wave pass from medium $\mathcal{M}_1$ into medium $\mathcal{M}_2$, the relationship $n_1 \sin \theta_1 = n_2 \sin \theta_2$ holds, where $\theta_1$ is the angle of incidence on the separation boundary between the two media and $\theta_2$ is the angle of refraction. A clarifying illustration of the refraction phenomenon for a light ray is shown in Fig. \ref{fig:snell1}. The angle of incidence $\theta_1$ is the angle between the incoming ray and the normal to the separation boundary. On the other hand, the angle of refraction $\theta_2$ is the angle between the outgoing ray and the normal to the separation boundary. 

\begin{figure}[h]
\begin{center}
\begin{tikzpicture}[scale=0.67]

    \coordinate (O) at (0,0) ;
    \coordinate (A) at (0,4) ;
    \coordinate (B) at (0,-4) ;
    \coordinate (a) at (-130:-5.2) ;
    \coordinate (b) at (50:5.05) ;
    \coordinate (b2) at (230:2.525) ;
    \coordinate (d) at (250:4.24) ;
    \coordinate (d2) at (250:2.7) ;
    \coordinate (d3) at (250:4.3) ;
    \coordinate (a2) at (50:2.6) ;
    \coordinate (c) at (0:0) ;
    \coordinate (c2) at (70:-4.24) ;
    \node[right] at (-2,2) {$\mathcal{M}_2$};
    \node[left] at (2,-2) {$\mathcal{M}_1$};
    \draw[dash pattern=on5pt off3pt] (A) -- (B) ;
\draw (-4,0) -- (4,0);

\path[->, thick] (O) edge (b); 
\draw [thick] (0,0) -- (d);

\draw[thick, dotted] (0,0) -- (70:4.24);
    \draw (0,1.4) arc (90:70:1.4) ;
    \node[] at (80:1.9)  {$\theta_{1}$};
    \draw (0,2.4) arc (90:50:2.4);
    \node[] at (60:2.9)  {$\theta_{2}$};

      \path[->, thick] (d) edge (d2); 
      
      \node[] at (175:3)  {boundary};
      \node[] at (241:4.1)  {light};
      \node[] at (103:3.75)  {normal};

\end{tikzpicture}
\end{center}
\caption{A light ray obeying Snell's law on the plane.\label{fig:snell1}}
\end{figure}
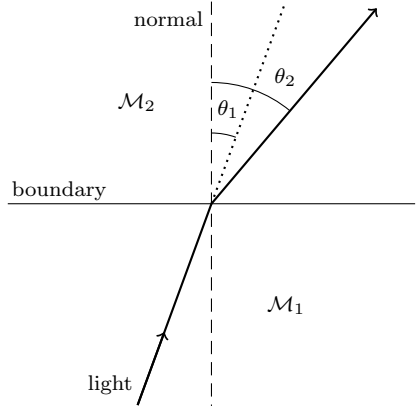

\MRM{In the following result we formally state the classical Snell's law  using the notation and concepts introduced for Problem \eqref{SPP}.}

\begin{prop}[Snell's law]
 \label{classicalSnell}
Assume $d=2$, $m=2$, $P_1=\{x\in\mathbb{R}^2: (\alpha_1,\alpha_2)x   \leq \beta\}$ with $\alpha_1,\alpha_2,\beta\in \mathbb{R}$, $P_2=\{x\in\mathbb{R}^2: (\alpha_1,\alpha_2)x   \geq \beta\}$, $F_{12}=\{x\in\mathbb{R}^2: (\alpha_1,\alpha_2) x = \beta\}$, $p_1=p_2=2$, $s=1$, $t=2$, $x_s\in P_1\setminus F_{12}$ and $x_t\in P_2\setminus F_{12}$. Let $a=(a_1,a_2)^T$, $b=(b_1,b_2)^T$ and $c=(c_1,c_2)^T$ be the consecutive breaking points of the unique geodesic path in Problem \eqref{SPP}, i.e., $a=x_s$, $c = x_t$ and $b$ is the gate point in $F_{12}$ that solves the problem. Then, for each non-zero vector $v\in L(F_{12})^\perp$, the gate point $b$ satisfies 
$$ \omega_s \sin \theta_s = \omega_t \sin \theta_t,$$
where  $\theta_s$ is  the angle between the vectors $b-a$ and $v$, and $\theta_t$ is  the angle between  the vectors $c-b$ and  $v$. 
\end{prop}
\begin{proof}
For a straightforward proof of this well-known result, we refer the reader to the proof of Lemma 3.2 in \cite{mitchellWRP}.\qed
\end{proof}  

We propose below two alternative statements of Snell's law based on equivalent \justo{statements} of Proposition \ref{classicalSnell}. The first one keeps the trigonometric interpretation of the law exchanging the sine function by the cosine one. The second one provides an interpretation of the law in terms of the dot product of vectors. As we will see, the subsequent extension of the law will follow in a clearer way from these alternative statements.

\begin{cor}[Snell's law - Cosine form]
 \label{cosineSnell}
\MRM{Assume the hypothesis of Proposition \ref{classicalSnell}.} Then, for each non-zero vector $v\in L(F_{12})$, the gate point $b$ satisfies 
$$ \omega_s \cos \phi_s = \omega_t \cos \phi_t,$$
where  $\phi_s$ is  the angle between the vectors $b-a$ and $v$, and $\phi_t$ is  the angle between  the vectors $c-b$ and  $v$. 
\end{cor}
\begin{proof}
The result follows directly from Proposition \ref{classicalSnell} and the complementary angle identity for trigonometric functions. \qed
\end{proof}

\begin{cor}[Snell's law - Dot product form]
\label{dotproductSnell}
\MRM{Assume the hypothesis of Proposition \ref{classicalSnell}.} Then, for each vector $v\in L(F_{12})$, the gate point $b$ satisfies 
$$ \omega_i \left( \frac{b-a}{\|b -a\|_{2}}\right)^T v= \omega_j \left( \frac{c-b}{\|c -b\|_{2}}\right)^T v.$$
\end{cor}
\begin{proof}
The result follows directly by simply noting in Corollary \ref{cosineSnell} that $(b-a)^T v = \|b-a\|_2\|v\|_2 \cos \phi_s$ and $(c-b)^T v = \|c-b\|_2\|v\|_2 \cos \phi_t$. \qed
\end{proof}

\MRM{The polarity correspondence between $\ell_p$-norms will be the key to understand the geometry behind the extension of Snell's law.} 

Polarity is an operation that induces a symmetric one-to-one correspondence in the class  of all the $\ell_p$-norms on $\mathbb{R}^d$ with $p\in (1,+\infty)$. Consider an $\ell_p$-norm with $p\in (1,+\infty)$ and let $B_p$ be its unit ball.  Then, it is known that there exists a unique $\ell_{p'}$-norm with $p'\in (1,+\infty)$ whose unit ball $B_{p'}$ is the polar set of $B_p$, where the polar set $B_p^\circ$ of $B_p$ is given by 
\begin{equation*}\label{polarset}
B_p^\circ = \left\{x'\in\mathbb{R}^d : x^T x'\leq 1, \forall x\in B_p\right\}.
\end{equation*}
This $\ell_{p'}$-norm is called the polar norm of $\ell_p$ or, alternatively, the dual norm of $\ell_p$. The polar norm $\ell_{p'}$ of $\ell_p$ can also be characterized as the support function of the unit ball $B_p$ of $\ell_p$, i.e.,
\begin{equation*}\label{supportfunction}
\|x\|_{p'} = \sup \{x^Tx' : x'\in B_p\}
\end{equation*}
for each $x\in\mathbb{R}^d$. As it is pointed out above, if $\ell_{p'}$ is the polar norm of $\ell_p$, then $\ell_p$ is the polar norm of $\ell_{p'}$ too. The following simple characterization of polarity correspondence holds: the norms $\ell_p$ and $\ell_{p'}$ are polar to each other iff $\frac{1}{p}+\frac{1}{p'}=1$ (see \cite{rockafellarCA} for more details on this and all the above statements). In what follows, for a given $p\in (1,+\infty)$, we will adopt the notation $p^{\circ}$ to refers to the real number in $(1,+\infty)$ such that $\frac{1}{p}+\frac{1}{p^\circ}=1$, i.e., $p^\circ = \frac{p}{p-1}$.

\MRM{For a given $p\in (1,+\infty)$, it is clear that the polar set $B_{p}^\circ$ is the counterpart in the normed space $(\mathbb{R}^d,\|\cdot\|_{p^\circ})$ of the set $B_p$ of the normed space $(\mathbb{R}^d,\|\cdot\|_p)$. To extends Snell's law to the framework induced by Problem \ref{SPP}, we will need to determine 
which vector in the normed space $(\mathbb{R}^d,\|\cdot\|_{p^\circ})$ is the counterpart of a given vector $v$ of the normed space $(\mathbb{R}^d,\|\cdot\|_p)$, in other words, which is its ``polar'' vector. Although the answer to this question may be well-known by specialist, we derive it in the following since we could not find any reference.}

\MRM{A simple way to characterize a vector of  $\mathbb{R}^d$ is by means of its length and its direction. For example, in the \justo{Euclidean} normed space $(\mathbb{R}^d, \|\cdot\|_2)$ two vectors  $v$ and $v'$ are the same vector iff they have the same length $\|v\|_2 = \|v'\|_2$ and the angle between then is $0$.}
 In this case, the angle between two non-zero vectors $v$ and $v'$
 is defined as the real number $\theta\in [0,\pi]$ satisfying the equality $\cos \theta = \frac{v^T v'}{\|v\|_2\|v'\|_2}$.
It is the Cauchy-Schwarz inequality $|v^T v'|\leq \|v\|_2\|v'\|_2$ which ensures the proper image of the cosine function in the above definition, i.e., $-1\leq \frac{v^Tv'}{\|v\|_2\|v'\|_2}\leq 1$. More generally, we recall that in all normed spaces $(\mathbb{R}^d, \|\cdot\|)$ where the norm $\|\cdot\|$ can be defined from an inner product $\langle\cdot,\cdot\rangle$ as $\|\tilde{v}\| = \sqrt{\langle \tilde{v}, \tilde{v}\rangle}$ for each $\tilde{v}\in\mathbb{R}^d$ (as it is the case of the \justo{Euclidean} norm with the dot product: $\|\tilde{v}\|_2 = \sqrt{\tilde{v}^T \tilde{v}}$ for each $\tilde{v}\in\mathbb{R}^d$), the Cauchy-Schwarz inequality $|\langle v, v'\rangle |\leq \|v\|\|v'\|$ is satisfied.  This leads to the known fact that in those normed spaces $(\mathbb{R}^d, \|\cdot\|)$ the angle between the vectors $v$ and $v'$ can be set as the real number $\phi\in [0,\pi]$ such that $\cos \phi = \frac{\langle v^T, v'\rangle}{\|v\|\|v'\|}$. Consider now a normed space $(\mathbb{R}^d, \|\cdot\|_p)$ with $p\in(1,+\infty)$. It is known that when $p\not= 2$ there is not an inner product $\langle \cdot, \cdot\rangle$ from which \MRM{the norm $\|\cdot\|_p$} can be defined as indicated above (see e.g. \cite{Kreyszig1978}). Moreover, the Cauchy-Schwarz inequality is not satisfied in $(\mathbb{R}^d, \|\cdot\|_p)$ when $p\not= 2$. \MRM{From the discussion above, it is clear that a vector $v$ of $(\mathbb{R}^d,\|\cdot\|_p)$ and its counterpart $v^{\circ}$ in $(\mathbb{R}^d,\|\cdot\|_{p^\circ})$
should have the same length $\|v\|_p = \|v^\circ\|_{p^\circ}$
(measured in their respective normed spaces) and the same direction according to an appropiate notion of angle. To define such an angle we base on H\"older inequality: $\sum_{k=1}^d |v_k v_k'| \leq \|v\|_p\|v'\|_{p^\circ}$, for all $v=(v_1,\cdots,v_d)^T,v'=(v_1',\cdots,v_d')^T\in\mathbb{R}^d$ (see e.g. \cite{Kreyszig1978}). H\"older inequality for $p=2$ implies Cauchy-Schwarz inequality, as the polar norm of the \justo{Euclidean} norm $\|\cdot\|_2$ is itself. In addition, H\"older inequality ensures $-1\leq \frac{v^Tv'}{\|v\|_p\|v'\|_{p^\circ}}\leq 1$ for all non-zero vectors $v,v'\in \mathbb{R}^d$, which makes possible to introduce the following notion of angle.}

\begin{defn}[$\ell_p$\hspace{0.1mm}-angle]\label{lpangle}
Let $p\in(1,+\infty)$. Given two non-zero vectors $v,v'\in \mathbb{R}^d$, the $\ell_p$-angle between $v$ and $v'$, which we denote by $\varphi_p (v,v')$, is the real number $\varphi_p (v,v')\in [0,\pi]$ such that $\cos \varphi_p (v,v') = \frac{v^T v'}{\|v\|_p\|v'\|_{p^\circ}}$.
\end{defn}

\MRM{We note that} due to the anisotropy of the $\ell_p$-norm when $p\not=2$, in general $\varphi_p (v,v')\not= \varphi_p (v',v)$, but it is satisfied  $\varphi_p (v,v')= \varphi_{p^{\circ}} (v',v)$.

\MRM{The notion of $\ell_p$-angle can also be related with the work \cite{Mangasarian1999}. In \cite{Mangasarian1999} it is studied 
the problem of projecting a point onto a hyperplane using an $\ell_p$-norm.} By \cite{Mangasarian1999} we know the following. Consider $v$ and $v'$ in the normed space $(\mathbb{R}^d,\|\cdot\|_p)$ with $p\in(1,+\infty)$ and let $\ell_p\text{-proj}_{\mathcal{H}} v$ be the projection of the point $v$ onto the hyperplane $\mathcal{H} = \{x\in\mathbb{R}^d : (v')^Tx = 0\}$.
Then, the distance between $v$ and $\ell_p\text{-proj}_{\mathcal{H}} v$ is given by $\|v - \ell_p\text{-proj}_{\mathcal{H}} v \|_p=  \frac{v^T v'}{\|v'\|_{p^\circ}}$. \MRM{On the other hand, it is easy to see} that the standard \MRM{Euclidean $\ell_2$-angle} $\theta$ between $v$ and $v'$ can be characterized by the relationship $\cos \theta = \frac{\|v - \ell_2\text{-proj}_{\mathcal{H}} v \|_2}{\|v\|_2}$. \MRM{The generalization of this characterization provides a natural interpretation of the $\ell_p$-angle as the angle such that $\cos \varphi_p(v,v') = \frac{\| v-\ell_p\text{-proj}_{\mathcal{H}} v \|_p}{\|v\|_p}$.}

\MRM{The proposition below stablishes a one-to-one correspondence between the vectors in $(\mathbb{R}^d,\|\cdot\|_{p})$ and the vectors in $(\mathbb{R}^d,\|\cdot\|_{p^\circ})$ that answers the question of how to define the ``polar'' vector.}

\begin{prop}\label{proppolarvector}
Assume $\mathbb{R}^d$ is endowed with an $\ell_p$-norm with $p\in(1,+\infty)$.
Consider the map between normed spaces $(\mathbb{R}^d,\|\cdot\|_{p}) \to (\mathbb{R}^d,\|\cdot\|_{p^\circ})$ given 
by $v=(v_1,\ldots,v_d)^T \mapsto v^\circ=(v_1^\circ,\ldots,v_d^\circ)^T$, where: if $v$ is not the zero vector, $v_k^\circ = \left(\frac{|v_k|}{\|v\|_p}\right)^{p-1}\sign(v_k) \|v\|_p$, $k=1,\ldots,d$; otherwise, $v^\circ$ is the zero vector. Then, given $v\in\mathbb{R}^d$, the vector $v^\circ$ is the unique vector in $\mathbb{R}^d$ satisfaying $\|v\|_p = \|v^\circ\|_{p^\circ}$ and $\varphi_p (v,v^\circ)=0$.
\end{prop}

\begin{proof}
Suppose first the case when $v\in\mathbb{R}^d$ is such that $\|v\|_p = 1$. Then, the coordinates of $v^\circ$ are $v_k^\circ = |v_k|^{p-1}\sign(v_k)$, $k=1,\ldots,d$. Now, consider the evaluation of the $\ell_p$-norm of $v$ given by the support function of the unit ball $B_{p^\circ}$ of $\ell_{p^\circ}$:
\begin{equation*}\label{sfunitvector}
\|v\|_{p} = \sup \{v^T v' : v'\in B_{p^\circ}\}.
\end{equation*}
It is not difficult to see that the above problem has a unique optimal solution. Note that if $v^\circ$ is that optimal solution and $\|v^\circ\|_{p^\circ}=1$, the result follows for this case.
But the above hypothesis is true since: on the one hand, 
\begin{equation*}
1 = \|v\|_{p} = \|v\|_{p}^p =  \sum_{k=1}^d |v_k|^p =  \sum_{k=1}^d v_k |v_k|^{p-1}\sign (v_k) = v^T v^\circ;
\end{equation*}
on the other hand,
\begin{equation*}
1 =v^T v^\circ= \sum_{k=1}^d v_k |v_k|^{p-1}\sign (v_k)  = \sum_{k=1}^d \left||v_k|^{p-1}\sign (v_k)\right|^{\frac{p}{p-1}} =\|v^\circ\|_{p^\circ}^{p^\circ} = \|v^\circ\|_{p^\circ}.
\end{equation*}

Consider now the case when $v\in\mathbb{R}^d$ is a non-zero vector with arbitrary $\ell_p$-norm. It is straightforward to check that $\left(\frac{v}{\|v\|_p}\right)^\circ = \frac{v^\circ}{\|v\|_p}$. As $\left\|\frac{v}{\|v\|_p}\right\|_p = 1$, by the reasoning above, we have that $\left\|\frac{v}{\|v\|_p}  \right\|_{p}=\left\|\frac{v^\circ}{\|v\|_p}\right\|_{p^\circ}$, from which follows that $\|v\|_p = \|v^\circ\|_{p^\circ}$. Also by the reasoning above, it holds that $\left(\dfrac{v}{\|v\|_p}\right)^T \dfrac{v^\circ}{\|v\|_{p}}=1$. Regarding the unicity of $v^\circ$, note that this is implied by the unicity of $\left(\frac{v}{\|v\|_p}\right)^\circ$.

Finally, if $v\in\mathbb{R}^d$ is the zero vector, the result follows directly. \qed
\end{proof}

\begin{defn}[Polar vector]\label{defpolarvector}
Assume $\mathbb{R}^d$ is endowed with an $\ell_p$-norm with $p\in(1,+\infty)$ and let $v\in\mathbb{R}^d$. The polar vector of $v$ is the vector $v^\circ$ given in Proposition \ref{proppolarvector}.
\end{defn}

We remark that the above definition of polar vector $v^\circ$ of \justo{of another vector} $v$ depends on the normed space $(\mathbb{R}^d,\|\cdot\|_p)$ in which $v$ is considered. In the following, we may avoid to explicitly point out that information when it is clear by the context.   

Now we are in position to state the anticipated extension of  Snell's law.

\begin{cor}[Snell's-like local optimality criterion for gate points - Dot product form]
\label{dotProductSnellPropLOCCP}
\textcolor{white}{lineskip} Assume $1<p_i<+\infty$, $i=1,...,m$.
Let $a=(a_1,\ldots,a_d)^T$, $b=(b_1,\ldots,b_d)^T$ and $c=(c_1,\ldots,c_d)^T$ be consecutive breaking points of a geodesic path in Problem \eqref{SPP} such that $a\in P_i\setminus F_{ij}$, $b\in F_{ij}$ and  $c\in P_j\setminus F_{ij}$. Let $\mathcal{F}_b$ be the face of $F_{ij}$ such that $b\in \text{\normalfont relint}(\mathcal{F}_b)$. Then, for each vector $v\in L(\text{\normalfont aff}(\mathcal{F}_b))$, the gate point $b$ satisfies 
$$ \omega_i \left( \left[ \frac{b-a}{\|b -a\|_{p_i}}\right]^\circ\right)^T v= \omega_j \left( \left[ \frac{c-b}{\|c -b\|_{p_j}}\right]^\circ\right)^T v,$$
or equivalently, 
$$ \omega_i \left(  \frac{(b-a)^\circ}{\|(b -a)^\circ\|_{p_i^\circ}}\right)^T v= \omega_j \left( \frac{(c-b)^\circ}{\|(c -b)^\circ\|_{p_j^\circ}}\right)^T v.$$
\end{cor}
\begin{proof}
First, observe that by Definition \ref{defpolarvector}, we can rewrite the local optimality condition given in Proposition \ref{propLOCCP} as 
\begin{equation}\label{dotProductSnellPropLOCCPeq12}
\omega_i \left( \left[ \frac{b-a}{\|b -a\|_{p_i}}\right]^\circ\right)^T (f-g)= \omega_j \left( \left[ \frac{c-b}{\|c -b\|_{p_j}}\right]^\circ\right)^T (f-g),
\end{equation}
for all $f,g\in\mathcal{F}_b$. Now, consider a vector $v\in L(\text{\normalfont aff}(\mathcal{F}_b))$ and let $x,x'\in \text{\normalfont aff}(\mathcal{F}_b)$ be any two points such that   $v=x-x'$. Note that, as $b\in \text{relint}(\mathcal{F}_b)$, there exist $\tilde{x}\in \mathcal{F}_b$ and $\eta\in \mathbb{R}$ such that $\eta(\tilde{x}-b) = x-x'$. Thus, the result follows by multiplying equality \eqref{dotProductSnellPropLOCCPeq12} by $\eta$ and taking $f=\tilde{x}$ and $g=b$.\qed
\end{proof}

\begin{cor}
Corollary \ref{dotProductSnellPropLOCCP} implies Corollary \ref{dotproductSnell}. 
\end{cor}
\begin{proof}
\MRM{Note that} there always exists \justo{a} square $R\subseteq \mathbb{R}^2$ such that Problem \eqref{SPP} stated in Corollary \ref{dotproductSnell} for polyhedra $P_1$, $P_2$ and $F_{12}$ is equivalent to the corresponding problem with polytopes $P_1':=P_1\cap R$, $P_2':=P_2\cap R$ and $F_{12}':=F_{12}\cap R$, where $b\in\text{int}(F_{12}')$ is also satisfied; it holds 
$L(F_{12}') = L(F_{12})$; the polar norm of the \justo{Euclidean} norm $\ell_2$ on $\mathbb{R}^d$ is itself and $v=v^\circ$ for each vector $v$ in $(\mathbb{R}^d,\|\cdot\|_2 )$.\qed
\end{proof}

Dot product form of Snell's law in Corollary \ref{dotproductSnell} states that, for any vector that can be consider in the separation boundary between the two ``media'', one has that the weighted dot product between this vector and the unit vector in the direction of the incoming ray is equal to the weighted dot product between the vector and the unit vector in the direction of the outgoing ray. Corollary \ref{dotProductSnellPropLOCCP} reveals that it is the polarity correspondence which is behind that relationship, although it is unnoticed due to the ``auto-polarity'' in $(\mathbb{R}^d,\|\cdot\|_2 )$. Specifically, $\left[ \frac{b-a}{\|b -a\|_{p_i}}\right]^\circ$ in Corollary \ref{dotProductSnellPropLOCCP} represents the polar of the unit vector in the direction of the incoming ray (equivalently, the unit vector in the direction of the polar of the incoming ray $\frac{(b-a)^\circ}{\|(b -a)^\circ\|_{p_i^\circ}}$)  
and $\left[ \frac{c-b}{\|c -b\|_{p_j}}\right]^\circ$ the polar of the unit vector in the direction of the outgoing ray (equivalently, the unit vector in the direction of the polar of the outgoing ray $\frac{(c-b)^\circ}{\|(c-b)^\circ\|_{p_j^\circ}}$).

We also provide a trigonometric interpretation of the generalized Snell's law based on the notion of $\ell_p$-angle.  

\begin{cor}[Snell's-like local optimality criterion for gate points - Cosine form]
\label{cosineSnellPropLOCCP}
\textcolor{white}{lineskip} \MRM{Assume the hypothesis of Corollary \ref{dotProductSnellPropLOCCP}.} Then, for each vector $v\in L(\text{\normalfont aff}(\mathcal{F}_b))$, the gate point $b$ satisfies
$$ \omega_i \|v\|_{p_i}\cos \varphi_{p_i^\circ}((b-a)^\circ, v) = \omega_j \|v\|_{p_j}\cos \varphi_{p_j^\circ}((c-b)^\circ, v) ,$$
or equivalently, 
$$ \omega_i \|v\|_{p_i}\cos \varphi_{p_i}(v,(b-a)^\circ) = \omega_j \|v\|_{p_j}\cos \varphi_{p_j}(v,(c-b)^\circ).$$ 
\end{cor}
\begin{proof}
The result follows directly from Corollary \ref{dotProductSnellPropLOCCP} and Definition \ref{lpangle}. \qed
\end{proof}

The results above establish that, \justo{analogously as  geodesic paths in the WRP ``obey'' Snell's law at the intermediate face of two polyhedra} (see Proposition 3.5 (b) in \cite{mitchellWRP}), the geodesic paths in Problem \eqref{SPP} ``obey'' the generalized Snell's law.

\subsection{MISOCP formulations}
\label{ss:22}

\MRM{Our solution scheme consists in representing Problem (\ref{SPP}) as a mixed-integer second order cone problem (MISOCP). To obtain the MISCOP representation, we first derive a mathematical programming formulation for Problem (\ref{SPP}) which is mixed integer linear except for the norms used to model distances. As it is shown below, such a formulation can be rewritten as a MISOCP using the results in \cite{BPE2014,refraction2017}.

Let $G^{\text{D}}=(V,A)$ be the  the directed graph associated to the graph $G=(V,E)$, where $A=\{(i,j),(j,i): \{i,j\}\in E\}$.} 

\begin{theorem}
\label{theo2}
Problem \eqref{SPP} is equivalent to the following problem,  \ML{called} \normalfont(SPP-F1):
\begin{subequations}
\label{PB2}
\begin{align}
\min & \sum_{i\in V} \omega_i d_{i}\label{c:2a}\\
\mbox{s.t. }& \sum_{(s,j)\in A} z_{sj} - \sum_{(h,s)\in A} z_{hs} = 1, \qquad\label{c:2b}\\
& \sum_{(i,j)\in A} z_{ij} - \sum_{(h,i)\in A} z_{hi}= 0, \qquad\forall i\in V\setminus \{s,t\},\label{c:2c}\\
& \sum_{(t,j)\in A} z_{tj} -\sum_{(h,t)\in A} z_{ht} = -1, \qquad\label{c:2d}\\
& \sum_{(h,i)\in A} z_{hi} \leq  1, \qquad\forall i\in V,\label{c:2e}\\
& \sum_{(i,j)\in A} z_{ij} \leq  1, \qquad\forall i\in V,\label{c:2f}\\
& d_{s} \geq \|\sum_{(s,j)\in A} \sum_{e\in \textnormal{Ext}(F_{sj})}\lambda_{sje} e - x_s\|_{p_s}, \qquad\label{c:2g}\\
& d_{i} \geq \|\sum_{(i,j)\in A} \sum_{e\in \textnormal{Ext}(F_{ij})}\lambda_{ije} e - \sum_{(h,i)\in A} \sum_{e\in \textnormal{Ext}(F_{hi})}\lambda_{hie} e\|_{p_i}, \qquad\forall i\in V\setminus\{s,t\},\label{c:2h}\\
& d_{t} \geq \|x_t - \sum_{(h,t)\in A} \sum_{e\in \textnormal{Ext}(F_{ht})}\lambda_{hte} e\|_{p_t}, \qquad\label{c:2i}\\
& \sum_{e\in \textnormal{Ext}(F_{ij})}\lambda_{ije} = z_{ij}, \qquad\forall (i,j)\in A,\label{c:2j}\\
& d_{i}\geq 0, \qquad\forall i\in V,\label{c:2k}\\
& z_{ij}\in\{0,1\}, \qquad\forall (i,j)\in A,\label{c:2l}\\
& \lambda_{ije}\geq 0, \qquad\forall (i,j)\in A,e\in\textnormal{Ext}(F_{ij}).\label{c:2m}
\end{align}
\end{subequations}
\end{theorem}

\begin{proof}

The domain of the decision variables in Problem (SPP-F1) is stated in \eqref{c:2k}-\eqref{c:2m}.

Variable $z_{ij}\in\{0,1\}$ indicates if arc $(i,j)\in A$ is chosen ($z_{ij}=1$) or not ($z_{ij}=0$) in a path from $s$ to $t$ in $G^{\text{D}}$. Then, constraints \eqref{c:2b}-\eqref{c:2d} are the standard flow conservation constraints to determine such a path. Constraints \eqref{c:2e} and \eqref{c:2f} are added to ensure that the path found is simple.

\ML{If $z_{ij}= 1$ then the path goes through a gate point in $F_{ij}$ and the corresponding $\lambda_{ije}$ determine this point. Then, (4e) and (4f)  together with (4j) imply that (4g)-(4i) provide the right value of the distance travelled by the path inside polytope $P_i$ if it is visited.}
\qed
\end{proof}

\MRM{Another formulation for Problem \eqref{SPP} is shown in Theorem \ref{theo3} below. This second formulation is derived using additional auxiliary variables.} To simplify the proof of Theorem \ref{theo3} we first present a technical lemma.

\begin{lem} \label{techlemma2}
Let $I,J$ be index sets. Let $\alpha_{ij}\geq 0$, $\beta_{i}\geq 0$, $\delta_{j}\geq 0$, $i\in I,j\in J$, be decision variables where variables $\beta_i$ are such that $\sum_{i\in I} \beta_i\leq 1$. Then, the set of constraints
\begin{subequations}
\begin{align}
&  \sum_{j\in J} \alpha_{ij} =  \beta_{i}, \qquad \forall i\in I,\label{c:98a}\\
&  \sum_{i\in I} \alpha_{ij} =  \delta_{j}, \qquad \forall j\in J,\label{c:98b}
\end{align}
\end{subequations}
imply
\begin{subequations}
\begin{align}
&  \alpha_{ij} \leq  \beta_{i}, &\forall i\in I,j \in J,\label{c:99a}\\
&  \alpha_{ij} \leq  \delta_{j},  &\forall i\in I,j \in J,\label{c:99b}\\
&  \alpha_{ij} \geq \beta_{i} +\delta_{j}-1, &\forall i\in I,j \in J.\label{c:99c}
\end{align}
\end{subequations}
\end{lem}
\begin{proof}
 It is straightforward that constraints \eqref{c:98a}-\eqref{c:98b} imply \eqref{c:99a}-\eqref{c:99b}. To see that they also imply \eqref{c:99c}, let us consider $i'\in I,j\in J$. Since $\sum_{i\in I} \beta_i\leq 1$, we have $1- \delta_j \geq \sum_{i\in I} \beta_i - \delta_j$. On the other hand, by \eqref{c:98b}, we have $\sum_{i\in I} \beta_i - \delta_j = \sum_{i\in I} \beta_{i}-\sum_{i \in I} \alpha_{ij} =  \sum_{i\in I} (\beta_{i}- \alpha_{ij})$. Moreover, by \eqref{c:98a}, all the terms of the last summation are nonnegatives, and therefore $\sum_{i\in I} (\beta_{i}- \alpha_{ij})\geq \beta_{i'} - \alpha_{i'j}$. Taking into account the above, we can conclude that $1- \delta_j \geq \alpha_{i'j}- \beta_{i'}$, which is equivalent to \eqref{c:99c} for $i'\in I, j \in J$. \qed
\end{proof}

\begin{theorem}
\label{theo3}
Problem \eqref{SPP} is equivalent to the following problem, \MRM{called} \normalfont(SPP-F2):
\begin{subequations}
\label{PB4}
\begin{align}
\min & \sum_{i\in V} \omega_i d_{i}\label{c:4a}\\
\mbox{s.t. }& \sum_{(s,j)\in A} z_{sj} - \sum_{(h,s)\in A} z_{hs} = 1, \qquad \label{c:4b}\\
& \sum_{(i,j)\in A} z_{ij} - \sum_{(h,i)\in A} z_{hi}= 0, \qquad \forall i\in V\setminus \{s,t\},\label{c:4c}\\
& \sum_{(t,j)\in A} z_{tj} -\sum_{(h,t)\in A} z_{ht} = -1, \qquad \label{c:4d}\\
& \sum_{(h,i)\in A} z_{hi} \leq  1, \qquad \forall i\in V,\label{c:4e}\\
& \sum_{(i,j)\in A} z_{ij} \leq  1, \qquad \forall i\in V,\label{c:4f}\\
& d_{s} \geq \sum_{(s,j)\in A}\| \sum_{e\in \textnormal{Ext}(F_{sj})}\lambda_{sje} e - x_s z_{sj}\|_{p_s}, \qquad \label{c:4g}\\
& d_{i} \geq \sum_{(h,i)\in A}\sum_{(i,j)\in A}\|\sum_{e\in \textnormal{Ext}(F_{ij})}\Psi_{hije} e - \sum_{e\in \textnormal{Ext}(F_{hi})}\Phi_{hije} e\|_{p_i}, \qquad \forall i\in V\setminus\{s,t\},\label{c:4h}\\
& d_{t} \geq \sum_{(h,t)\in A} \|x_t z_{ht} - \sum_{e\in \textnormal{Ext}(F_{ht})}\lambda_{hte} e\|_{p_t}, \qquad \label{c:4i}\\
& \sum_{e\in \textnormal{Ext}(F_{ij})}\lambda_{ije} = z_{ij}, \qquad \forall (i,j)\in A,\label{c:4j}\\
&  \sum_{(i,j)\in A} \rho_{hij} =  z_{hi}, \qquad \forall i\in V\setminus\{s,t\},(h,i)\in A,\label{c:4k}\\
&  \sum_{(h,i)\in A} \rho_{hij} =  z_{ij}, \qquad \forall i\in V\setminus\{s,t\},(i,j)\in A,\label{c:4l}\\
&  \sum_{e\in \text{Ext}(F_{hi})} \Phi_{hije} =  \rho_{hij}, \qquad \forall i\in V\setminus\{s,t\},(h,i),(i,j)\in A,\label{c:4m}\\
&  \sum_{(i,j)\in A} \Phi_{hije} =  \lambda_{hie}, \qquad \forall i\in V\setminus\{s,t\},(h,i)\in A,e \in \text{Ext}(F_{hi}),\label{c:4n}\\
&  \sum_{e\in \text{Ext}(F_{ij})} \Psi_{hije} =  \rho_{hij}, \qquad \forall i\in V\setminus\{s,t\},(h,i),(i,j)\in A,\label{c:4o}\\
&  \sum_{(h,i)\in A} \Psi_{hije} =  \lambda_{ije}, \qquad \forall i\in V\setminus\{s,t\},(i,j)\in A,e \in \text{Ext}(F_{ij}),\label{c:4p}\\
& d_{i}\geq 0, \qquad \forall i\in V,\label{c:4q}\\
& z_{ij}\in\{0,1\}, \qquad \forall (i,j)\in A,\label{c:4r}\\
& \lambda_{ije}\geq 0, \qquad \forall (i,j)\in A, e\in\textnormal{Ext}(F_{ij}),\label{c:4s}\\
& \rho_{hij}\geq 0, \qquad \forall i\in V\setminus\{s,t\},(h,i),(i,j)\in A,\label{c:4t}\\
& \Phi_{hije}\geq 0, \qquad \forall i\in V\setminus\{s,t\},(h,i),(i,j)\in A, e\in \textnormal{Ext}(F_{hi}),\label{c:4u}\\
& \Psi_{hije}\geq 0, \qquad \forall i\in V\setminus\{s,t\},(h,i),(i,j)\in A, e\in \textnormal{Ext}(F_{ij}).\label{c:4v}
\end{align}
\end{subequations}
\end{theorem}

\begin{proof}
\MRM{Since variables $z_{ij}$ codify a simple path from $s$ to $t$, observe that Problem \eqref{SPP} is equivalent to the problem}
\begin{align}
\min &  \sum_{i\in V} \omega_i d_i \label{SPPAuxProblem4}  \\
\mbox{s.t. }& \eqref{c:4b} \text{-} \eqref{c:4f},\nonumber\\
& d_{s} \geq \sum_{(s,j)\in A}\|\sum_{e\in \textnormal{Ext}(F_{sj})}\lambda_{sje} e - x_s\|_{p_s}z_{sj}, &\nonumber\\
& d_{i} \geq \sum_{(h,i)\in A}\sum_{(i,j)\in A}\|\sum_{e\in \textnormal{Ext}(F_{ij})}\lambda_{ije} e - \sum_{e\in \textnormal{Ext}(F_{hi})}\lambda_{hie} e\|_{p_i}z_{hi}z_{ij}, &\forall i\in V\setminus\{s,t\},\nonumber\\
& d_{t} \geq \sum_{(h,t)\in A}\|x_t - \sum_{e\in \textnormal{Ext}(F_{ht})}\lambda_{hte} e\|_{p_t}z_{ht}, &\nonumber\\
& \eqref{c:4j},\eqref{c:4q} \text{-} \eqref{c:4s}.\nonumber 
\end{align}

\MRM{Let us see that Problem \eqref{SPPAuxProblem4} is equivalent to Problem (SPP-F2).} In Problem \eqref{SPPAuxProblem4}, the equivalence of the following expressions holds:
\begin{align}
& \sum_{(s,j)\in A}\| \sum_{e\in \textnormal{Ext}(F_{sj})}\lambda_{sje} e - x_s \|_{p_s} z_{sj} = \sum_{(s,j)\in A}\| \sum_{e\in \textnormal{Ext}(F_{sj})}\lambda_{sje} z_{sj}e - x_s z_{sj}\|_{p_s} = \sum_{(s,j)\in A}\| \sum_{e\in \textnormal{Ext}(F_{sj})}\lambda_{sje} e - x_s z_{sj}\|_{p_s}, \nonumber\\
& \sum_{(h,t)\in A}\|x_t - \sum_{e\in \textnormal{Ext}(F_{ht})}\lambda_{hte} e\|_{p_t}z_{ht} = \sum_{(h,t)\in A} \|x_t z_{ht} - \sum_{e\in \textnormal{Ext}(F_{ht})}\lambda_{hte} z_{ht}e\|_{p_t} = \sum_{(h,t)\in A} \|x_t z_{ht} - \sum_{e\in \textnormal{Ext}(F_{ht})}\lambda_{hte} e\|_{p_t}, \nonumber 
\end{align}
by homogeneity of the norms and constraints \eqref{c:2j}.
Note that this equivalence leads to constraints \eqref{c:4g} and \eqref{c:4i} in Problem (SPP-F2). On the other hand, for each $i\in V\setminus\{s,t\}$, it follows:  
$$\sum_{(h,i)\in A}\sum_{(i,j)\in A}\|\sum_{e\in \textnormal{Ext}(F_{ij})}\lambda_{ije} e - \sum_{e\in \textnormal{Ext}(F_{hi})}\lambda_{hie} e\|_{p_i}z_{hi}z_{ij}=\sum_{(h,i)\in A}\sum_{(i,j)\in A}\|\sum_{e\in \textnormal{Ext}(F_{ij})}\lambda_{ije}z_{hi}z_{ij} e - \sum_{e\in \textnormal{Ext}(F_{hi})}\lambda_{hie}z_{hi}z_{ij} e\|_{p_i}.$$
The resulting products of three variables in the argument of the norms above can be linearized as it is  shown below. This linearization will lead to constraints \eqref{c:4h} and it will complete the proof.  

Consider  $i\in V\setminus\{s,t\}$, $(h,i),(i,j)\in A$. In Problem (SPP-F2), we represent the product $z_{hi}z_{ij}$ by means of variable $\rho_{hij}\geq 0$. The product $\lambda_{hie}z_{hi}z_{ij} = \lambda_{hie}\rho_{hij}$ is represented with variable $\Phi_{hije}\geq 0$, for each $e\in\text{Ext}(F_{hi})$. On the other hand, the product $\lambda_{ije}z_{hi}z_{ij} = \lambda_{ije}\rho_{hij}$ is represented with variable $\Psi_{hije}\geq 0$, for each $e\in\text{Ext}(F_{ij})$. We can set the values of the variables $\rho_{hij},\Phi_{hije},\Psi_{hije}$ using the following set of constraints:
\begin{subequations}
\begin{align}
&  \rho_{hij} \leq  z_{hi}, &\forall i\in V\setminus\{s,t\},(h,i),(i,j)\in A,\label{c:3k}\\
&  \rho_{hij} \leq  z_{ij}, &\forall i\in V\setminus\{s,t\},(h,i),(i,j)\in A,\label{c:3l}\\
&  \rho_{hij} \geq  z_{hi}+z_{ij}-1, &\forall i\in V\setminus\{s,t\},(h,i),(i,j)\in A, \label{c:3m}\\
&  \Phi_{hije} \leq  \rho_{hij}, &\forall i\in V\setminus\{s,t\},(h,i),(i,j)\in A, e\in \textnormal{Ext}(F_{hi}),\label{c:3n}\\
&  \Phi_{hije} \leq  \lambda_{hie},  &\forall i\in V\setminus\{s,t\},(h,i),(i,j)\in A, e\in \textnormal{Ext}(F_{hi}),\label{c:3o}\\
&  \Phi_{hije} \geq \lambda_{hie} +\rho_{hij}-1, &\forall i\in V\setminus\{s,t\},(h,i),(i,j)\in A, e\in \textnormal{Ext}(F_{hi}),\label{c:3p}\\
&  \Psi_{hije} \leq  \rho_{hij}, &\forall i\in V\setminus\{s,t\},(h,i),(i,j)\in A, e\in \textnormal{Ext}(F_{ij}),\label{c:3q}\\
&  \Psi_{hije} \leq  \lambda_{ije}, &\forall i\in V\setminus\{s,t\},(h,i),(i,j)\in A, e\in \textnormal{Ext}(F_{ij}),\label{c:3r}\\
&  \Psi_{hije} \geq \lambda_{ije} +\rho_{hij}-1,&\forall i\in V\setminus\{s,t\},(h,i),(i,j)\in A, e\in \textnormal{Ext}(F_{ij}). \label{c:3s}
\end{align}
\end{subequations}
Constraints \eqref{c:3k}-\eqref{c:3s} refer to a standard way of linearizing the product of variables. However, when integrated in our problem, they are dominated by constraints \eqref{c:4k}-\eqref{c:4p}. In particular, by Lemma \ref{techlemma2}, constraints \eqref{c:4k}-\eqref{c:4l}, \eqref{c:4m}-\eqref{c:4n} and \eqref{c:4o}-\eqref{c:4p} imply 
\eqref{c:3k}-\eqref{c:3m}, \eqref{c:3n}-\eqref{c:3p} and \eqref{c:3q}-\eqref{c:3s}, respectively. Therefore, constraints \eqref{c:4k}-\eqref{c:4p} lead to a stronger formulation than \eqref{c:3k}-\eqref{c:3s}.\qed
\end{proof}

\MRM{The above formulations can be rewritten as MISOCPs. The reformulation is possible because the norms in constraints \eqref{c:2g}-\eqref{c:2i} of Problem (SPP-F1), and the norms in constraints \eqref{c:4g}-\eqref{c:4i} of Problem (SPP-F2), can be represented as a set of second order cone constraints.} The procedure to construct these constraints is detailed in \cite{BPE2014,refraction2017}. However, in the following, we outline  the main ideas of this construction for the sake of completeness.

Assuming $p = \frac{q}{r}$ with $q, r\in\N\setminus\{0\}$, $q>r$ and $\gcd(q,r)=1$, any constraint in the form $Z \geq \|X-Y\|_{p}$, where $X$ and $Y$ are variable vectors in $\R^d$, can be equivalently rewritten as (see Lemma 9 in \cite{refraction2017}):

\begin{minipage}{\textwidth}
\begin{minipage}{0.85\textwidth}
$$\hspace*{3cm}
\left.\begin{array}{ll}
U_{k} + X_k - Y_{k}\ge 0,&\; k=1, \ldots, d, \\
U_{k} - X_k + Y_{k}\ge 0,& \; k=1, \ldots, d,\\
U_{k}^{q} \leq W_{k}^{r} Z^{q-r},&  k=1, \ldots, d,\\
\dsum_{k=1}^d W_{k} \leq Z,& \\
W_k \geq 0, & k=1, \ldots, d.\end{array}\right\}
$$
\end{minipage}
\begin{minipage}{0.1\textwidth}
\begin{equation}\label{in:norm}
\end{equation}
\end{minipage}
\end{minipage}

\noindent Except for constraints $U_{k}^{q} \leq W_{k}^{r} Z^{q-r}$, $k=1, \ldots, d$, all the constraints in \eqref{in:norm} are linear.
Using Lemma 1 in \cite{BPE2014}, each one of those nonlinear constraints can be represented as a set of at most $2\log q$ constraints of the form $\alpha^2 \leq \beta \delta$. Finally, note that inequality $\alpha^2 \leq \beta \delta$ is equivalent to inequality $\left\| \begin{pmatrix} 2\alpha\\\beta-\delta \end{pmatrix}\right\|_2 \leq \beta+\delta$, which can be seen as a second order cone constraint.   
 
We comment now the case when the norm in $Z \geq \|X-Y\|$ is polyhedral, as for instance the $\ell_1$ or $\ell_{\infty}$ norms. In that case, it is well-known that $Z \geq \|X-Y\|$ can be stated as a set of linear inequalities (see e.g. \cite{NP05,NPC03,WardWendell1985}). Let polytope $P$ be the unit ball of the norm $\|\cdot\|$. Then, $Z \geq \|X-Y\|$ is equivalent to 
$$ Z \geq e^T (X-Y), \; \forall e \in {\rm Ext}(P^\circ),$$ 
where $P^\circ$ is the polar set of $P$. \MRM{Thus, the above formulations can also handle polyhedral norms.}

\MRM{We have derived two different formulations for Problem \eqref{SPP}. Note that the number of norms in formulation (SPP-F1) is smaller than the one in formulation (SPP-F2). Therefore,} the number of variables needed for the MISOCP representation of the norms in  formulation (SPP-F1) is smaller than the one nedeed for formulation (SPP-F2). Furthermore, without taking into account those variables, the set of decision variables of formulation (SPP-F1) is contained in the one of formulation (SPP-F2). \justo{In summary, formulation (SPP-F1) has the advantage of having a smaller size than formulation (SPP-F2). Nevertheless, formulation (SPP-F2) outperforms formulation (SPP-F1) in that it provides a tighter continuous relaxation. This result is established below.}

\begin{prop}
\label{propRelaxation}
Let $\zeta,\zeta'\geq 0$ be the objective values of the continuous relaxations of formulation \textnormal{(SPP-F1)} and formulation \textnormal{(SPP-F2)}, respectively. Then, $\zeta\leq\zeta'$. 
\end{prop}

\begin{proof}
Consider the continuous relaxations of Problem (SPP-F1) and Problem (SPP-F2). Note that the set of decision variables of the first problem is contained in the set of decision variables of the second one. Note also that the feasible domain for the set of variables $z_{ij}$ and $\lambda_{ije}$ in the first problem contains the corresponding feasible domain in the second problem. Let us assume that we are given a realization for this set of variables that is feasible in both problems. 

Then, since constraints \eqref{c:2b} and \eqref{c:2e} imply $\sum_{(s,j)\in A} z_{sj}=1$, and using  the triangular inequality, we have that
$$\|\sum_{(s,j)\in A} \sum_{e\in \textnormal{Ext}(F_{sj})}\lambda_{sje} e - x_s\|_{p_s} = \|\sum_{(s,j)\in A} (\sum_{e\in \textnormal{Ext}(F_{sj})}\lambda_{sje} e - x_s z_{sj})\|_{p_s} \leq  \sum_{(s,j)\in A}\| \sum_{e\in \textnormal{Ext}(F_{sj})}\lambda_{sje} e - x_s z_{sj}\|_{p_s}.$$
In a similar way, since constraints \eqref{c:2d} and \eqref{c:2f} imply $\sum_{(h,t)\in A} z_{ht}=1$, and  applying the triangular inequality, it follows that
$$ \|x_t - \sum_{(h,t)\in A} \sum_{e\in \textnormal{Ext}(F_{ht})}\lambda_{hte} e\|_{p_t} = \|\sum_{(h,t)\in A} (x_t z_{ht} -  \sum_{e\in \textnormal{Ext}(F_{ht})}\lambda_{hte} e)\|_{p_t} \leq \sum_{(h,t)\in A} \|x_t z_{ht} - \sum_{e\in \textnormal{Ext}(F_{ht})}\lambda_{hte} e\|_{p_t}.$$
On the other hand, by constraints \eqref{c:4n} and \eqref{c:4p}, and by the triangular inequality, for each $i\in V\setminus\{s,t\}$ we have that
$$\|\sum_{(i,j)\in A} \sum_{e\in \textnormal{Ext}(F_{ij})}\lambda_{ije} e - \sum_{(h,i)\in A} \sum_{e\in \textnormal{Ext}(F_{hi})}\lambda_{hie} e\|_{p_i}= \|\sum_{(i,j)\in A}\sum_{e\in \textnormal{Ext}(F_{ij})}\sum_{(h,i)\in A}\Psi_{hije} e - \sum_{(h,i)\in A}\sum_{e\in \textnormal{Ext}(F_{hi})}\sum_{(i,j)\in A}\Phi_{hije} e\|_{p_i}  $$
$$ = \|\sum_{(h,i)\in A} \sum_{(i,j)\in A}(\sum_{e\in \textnormal{Ext}(F_{ij})}\Psi_{hije} e - \sum_{e\in \textnormal{Ext}(F_{hi})}\Phi_{hije} e)\|_{p_i}\leq  \sum_{(h,i)\in A}\sum_{(i,j)\in A}\|\sum_{e\in \textnormal{Ext}(F_{ij})}\Psi_{hije} e - \sum_{e\in \textnormal{Ext}(F_{hi})}\Phi_{hije} e\|_{p_i}.$$

Finally we note that the above inequalities are preserved  
when the MISOCP representation of the norms is applied.
\qed
\end{proof}

\MRM{We compare the computational peformance of formulation (SPP-F1) and formulation (SPP-F2) in Section \ref{ss:23}.}

\subsubsection{Preprocessing}
\label{sss:221}

\justo{The continuous relaxation of formulation (SPP-F2) has shown to provide really tight lower bounds in our computational study. We take advantage of this observation to propose \MRM{an ad hoc preprocessing scheme for Problem \eqref{SPP}} to discard some regions that will not be used by shortest paths. 
\MRM{Let $\zeta_i$ be the objective value of the continuous relaxation of formulation (SPP-F2) but imposing that a shortest path must visit polyhedron $P_i$.} If $\zeta_i$ is greater than the objective value $\zeta^{\text{MIP}}$ of a feasible solution for Problem \eqref{SPP}, then polyhedron $P_i$ will not be visited for any shortest path. The implication is that all the variables associated to polyhedron $P_i$ in a formulation for Problem \eqref{SPP} can be removed. Since preprocessing all the $m$ regions in the problem can be computationally expensive, as it requires solving a SOCP problem for each polyhedron, in practice we only check a reduced number $m^*$  of the ``most promising'' polyhedra to be discarded. \MRM{Specifically, we process first the polyhedra $P_i$ furthest away from the terminal points $x_s,x_t$ according to the measure $\mathcal{D}(P_i) = \min_{x\in P_i} \{\ell_2(x_s,x) + \ell_2(x,x_t)\}$. 
The upper bound on the objective function $\zeta^{\text{MIP}}$  that we consider is the objective value of Problem \eqref{SPP} fixing the simple path $\gamma^*$ induced by the intersection of the segment $[x_s,x_t]$ with the set of polyhedra $\{P_1,...,P_m\}$.}}
We show below the basic pseudocode of the above preprocessing. The preprocessing algorithm returns a list $\mathcal{L}$ of polyhedra that are not visited in any shortest path for Problem \eqref{SPP}.

\begin{algorithm}
\caption{\MRM{Preprocessing for Problem \eqref{SPP}}}
\label{alg:preprocess}
\renewcommand{\thealgorithm}{}
\floatname{algorithm}{}
\begin{algorithmic}[1]
    \Begin
    \State compute the simple path $\gamma^*$ induced by the intersection of the segment $[x_s,x_t]$ with the set of polyhedra $\{P_1,...,P_m\}$
    \State compute the objective value  $\zeta^{\text{MIP}}$ of Problem \eqref{SPP} fixing the path $\gamma^*$
    \State sort polyhedra $P_{1},...,P_{m}$ in non-decreasing sequence with respect to $\mathcal{D}$: $P_{(1)},...,P_{(m)}$ such that $\mathcal{D}(P_{(1)})\geq...\geq \mathcal{D}(P_{(m)})$
    \State let $\mathcal{L}:=\emptyset$
    \For{$i=1,...,m^*$}
    \State compute $\zeta_{(i)}$ the objective value of the continuous relaxation of formulation (SPP-F2) but imposing that a shortest path must visit polyhedron $P_{(i)}$
    \If{$\zeta_{(i)}>\zeta^{\text{MIP}}$}
    \State  update $\mathcal{L} := \mathcal{L}\cup \{P_{(i)}\}$
    \EndIf
    \EndFor
    \State \Return $\mathcal{L}$
    \End
\end{algorithmic}

\end{algorithm}

\subsection{Computational experiments}
\label{ss:23}

\MRM{In order to evaluate the efficiency of the MISOCP formulations (SPP-F1) and (SSP-F2), as well as the preprocessing for Problem \eqref{SPP},} we have performed a series of computational experiments on a set of instances of the problem. The formulations have been coded in XPRESS solver version 8.9 keeping the default configuration for the different tolerances in the nonlinear module of the solver. The different instances have been generated using MATLAB R2018A which makes calls to XPRESS solver for solving the MISOCP programs. All experiments were run in a computer DellT5500 with a processor Intel(R) Xeon(R) with a CPU X5690 at 3.75 GHz and 48 GB of RAM memory.

We explain how we have generated the different instances of Problem \eqref{SPP}. We produce each instance of Problem \eqref{SPP} by randomly generating a polyhedral \justo{subdivision}. The process to generate the polyhedral \justo{subdivision} is as follows. We randomly generate $m$ points in the square $[0,10]^2$ and then compute the \justo{Euclidean} Voronoi diagram of these points in the square. The polyhedral \justo{subdivision} consists of the $m$ resulting \MRM{Voronoi cells}. Once generated the polyhedral \justo{subdivision} $\{P_1,...,P_m\}$, we \MRM{randomly} associate to each polytope one of the norms $\ell_1$, $\ell_{1.5}$, $\ell_2$, $\ell_3$ or $\ell_\infty$. \MRM{The weight associated to polytope $P_i$ is $\omega_i=1$, $i=1,...,m$. In all instances the starting and the terminal points are $x_s=(0,0)^T$ and $x_t=(10,10)^T$, respectively.} 

\MRM{We first compare the performance of formulations (SPP-F1) and (SPP-F2) on instances with a different number $m$ of polytopes. For each value of $m$, we randomly generate 5 subdivisions consisting of $m$ polytopes. We set a time limit for solving each instance of 7200 seconds and we report for each value of $m$:} 
\begin{itemize}
\item CPU(s): average, minimum and maximum CPU time over the 5 randomly generated instances.

\item Gap(\%): average, minimum and maximum gap over the 5 randomly generated instances.
\end{itemize} 
\MRM{The result obtained for formulation (SPP-F1) are shown in Table \ref{table-SPP-F1} and the ones obtained for formulation (SPP-F2) in Table \ref{table-SPP-F2}.}

\begin{table}[H]
\begin{center}
{\small
\begin{tabular}{|r|rrr|rrr|}
\hline
$m$ &  \multicolumn{3}{c|}{CPU(s)} & \multicolumn{3}{c|}{gap(\%)}  \\ \hline
         &	aver	&	min	&	max	&	aver	 & min	&	max		  \\ \hline
50 &  608.04 &  90.08 &  1962.01 &  0.01 &  0.00 &  0.01 \\
100 &  7200.00 &  7200.00 &  7200.00 &  3.62 &  1.84 &  6.79 \\
150 &  7200.00 &  7200.00 &  7200.00 &  10.14 &  7.55 &  11.95 \\
\hline
\end{tabular}
}
\end{center}
\caption{Results obtained for Formulation (SPP-F1).}\label{table-SPP-F1}
\end{table}

\begin{table}[H]
\begin{center}
{\small
\begin{tabular}{|r|rrr|rrr|}
\hline
$m$ &  \multicolumn{3}{c|}{CPU(s)} & \multicolumn{3}{c|}{gap(\%)}  \\ \hline
         &	aver	&	min	&	max	&	aver	 & min	&	max		  \\ \hline
50 &  32.27 &  17.54 &  46.34 &  0.00 &  0.00 &  0.01 \\
100 &  997.45 &  78.86 &  4037.64 &  0.01 &  0.00 &  0.01 \\
150 &  5572.75 &  164.27 &  7200.00 &  0.09 &  0.00 &  0.39 \\
200 &  3125.70 &  235.53 &  7200.00 &  0.06 &  0.00 &  0.26 \\
250 &  3951.26 &  1242.19 &  7200.00 &  0.11 &  0.00 &  0.44 \\
300 &  4537.31 &  335.25 &  7200.00 &  0.44 &  0.01 &  0.76 \\
350 &  4959.25 &  1206.76 &  7200.00 &  0.27 &  0.00 &  0.57 \\
400 &  7200.00 &  7200.00 &  7200.00 &  1.31 &  0.02 &  4.14 \\
450 &  6107.93 &  1700.82 &  7200.00 &  1.80 &  0.00 &  7.93 \\
500 &  6671.42 &  4511.06 &  7200.00 &  2.70 &  0.00 &  8.40 \\
\hline
\end{tabular}
}
\end{center}
\caption{Results obtained for formulation (SPP-F2).}\label{table-SPP-F2}
\end{table}

\MRM{We only run formulation (SPP-F1) for instances with up to 150 regions, since as it can be seen in Table \ref{table-SPP-F1} the formulation always reaches the time limit for this size and \justo{the gaps obtained for $m=150$ are already above $7.5\%$.} On the other hand, formulation (SPP-F2) has shown to be faster and we have run it for instances with up to 500 regions obtaining reasonable gaps. The computational dominance of formulation (SPP-F2) over formulation (SPP-F1) is explained by 
Proposition \ref{propRelaxation}: although for formulation (SPP-F1) the problems in the nodes of the branch-and-bound environment are solved faster (the size of these problems is smaller), a much lower number of nodes has to be explored in the case of formulation (SPP-F2) because} \justo{the tightness of the lower bounds provided by that formulation.}   

\MRM{We test now the preprocessing in Algorithm \ref{alg:preprocess}. Since as shown above formulation (SPP-F1) is outperformed by formulation (SPP-F2), we only test the effect of the preprocessing in the second formulation. For each instance, we decide to preprocess around the 10\% of the polytopes in the subdivision, i.e. in Algorithm \ref{alg:preprocess} we choose $m^*=\lceil \frac{m}{10}\rceil$. In order to perform a fair comparison between the performance of formulation (SPP-F2) with and without preprocessing, the preprocessing time is subtracted to the 7200 seconds assigned to the formulation for solving the instance. The results obtained for formulation (SPP-F2) applying the preprocessing are shown in Table \ref{table-SPP-F2-Pre}}.

\begin{table}[H]
\begin{center}
{\small
\begin{tabular}{|r|rrr|rrr|}
\hline
$m$ &  \multicolumn{3}{c|}{CPU(s)} & \multicolumn{3}{c|}{gap(\%)}  \\ \hline
         &	aver	&	min	&	max	&	aver	 & min	&	max		  \\ \hline
50 &  31.72 &  15.60 &  49.20 &  0.00 &  0.00 &  0.00 \\
100 &  175.34 &  157.02 &  206.57 &  0.00 &  0.00 &  0.01 \\
150 &  3813.09 &  258.64 &  7200.00 &  0.11 &  0.00 &  0.54 \\
200 &  3445.39 &  354.24 &  7200.00 &  0.01 &  0.00 &  0.03 \\
250 &  4918.68 &  1021.12 &  7200.00 &  0.21 &  0.00 &  0.65 \\
300 &  5028.60 &  705.18 &  7200.00 &  0.33 &  0.00 &  0.87 \\
350 &  6965.18 &  5999.88 &  7200.00 &  0.52 &  0.00 &  1.11 \\
400 &  7200.00 &  7200.00 &  7200.00 &  0.46 &  0.01 &  1.00 \\
450 &  6249.92 &  2406.81 &  7200.00 &  0.40 &  0.01 &  0.91 \\
500 &  6830.36 &  5312.42 &  7200.00 &  1.28 &  0.00 &  2.50 \\
\hline
\end{tabular}
}
\end{center}
\caption{Results obtained for formulation (SPP-F2) using the preprocessing in Algorithm \ref{alg:preprocess} with $m^*=\lceil \frac{m}{10}\rceil$.}\label{table-SPP-F2-Pre}
\end{table}

\MRM{Comparing Table \ref{table-SPP-F2} with Table \ref{table-SPP-F2-Pre} we conclude that, although the preprocessing strategy does not always reduce the running times, it reduces the gaps obtained within the fixed time limit in most of the instances, specially in those with a larger size.}

\subsection{Extensions}
\label{ss:24}

In this section we elaborate on different extensions that can be considered \MRM{for Problem (\ref{SPP})}. \ML{We  show that the mathematical programming formulations developed in the previous sections are also valid for these extensions provided that we modify the definition of the auxiliary graph.}

\subsubsection{Rapid transit boundaries}
\label{sss:241}

Consider again $\{P_1,...,P_m\}$, a \justo{subdivision} in polyhedra of $\mathbb{R}^d$ where each polyhedron $P_i$ is endowed with a different $\ell_{p_i}$-norm, $i=1,...,m$. Suppose now that each \MRM{maximal} common face $F_{ij}$ of two polyhedra $P_i$ and $P_j$ has also an associated $\ell_{p_{ij}}$-norm, possibly ``faster''  than both $\ell_{p_i}$ and $\ell_{p_j}$ (i.e. $p_{ij}>p_i,p_j$). Then, the common face of two polyhedra can be interpreted as a rapid transit boundary that can be used to reduce the length of the paths traversing both regions (see Section 4 in \cite{refraction2017}). \ML{This extension comes from the logistic and transportation literature  where it is used to model rapid transit lines \MRM{(see e.g. \cite{carrizosa-chia,NPC03,PuertoRCh2011})}.}

Let $\omega_i>0$ be the weight associated to polyhedra $P_{i}$ and $\omega_{ij}>0$  be the weight associated to polyhedra $F_{ij}$, for all different $i,j=1,...,m$ such that $P_i\cap P_j\neq \emptyset$. Consider  
the starting and the terminal points $x_s\in P_s$ and $x_t\in P_t$, $s\neq t$, and let the graph $G=(V,E)$ and the set $\Gamma$ be defined as in Section \ref{ss:21}. We can define the shortest path problem in this new setting as
\begin{align*}
D^{\text{RT}}(x_s,x_t) = \inf_{\scriptsize\begin{array}{c}(s,i_1,i_2,\cdots ,i_{k-1},i_k, t)\in \Gamma  \\ y_{s i_1}^1,y_{s i_1}^2\in F_{s i_1},\ldots,y_{i_kt}^1, y_{i_kt}^2\in F_{i_kt}\end{array}} &\omega_s\|x_s-y_{s i_1}^1\|_{p_{s}} + \omega_{s i_1}\|y_{si_1}^1-y_{s i_1}^2\|_{p_{s i_1}} \\&+ \omega_{i_1}\|y_{si_1}^2-y_{i_1 i_2}^1\|_{p_{i_1}} + \ldots +   \omega_{ i_k}\|y_{i_{k-1}i_k}^2-y_{i_k i_t}^1\|_{p_{i_k}}\\ \\ &+ \omega_{ i_k t}\|y_{i_k i_t}^1-y_{i_k i_t}^2\|_{p_{i_k t}} + \omega_{t}\|y_{i_k t}^2-x_t\|_{p_{t}}. \label{SPPRT} \tag{${\rm SPP^{RT}}$}
\end{align*}
\MRM{If $s= t$, then $D^{\text{RT}}(x_s,x_t) = \|x_t-x_s\|_{p_{s}}$.} 

\emph{MISOCP representation:} Let us see that Problem (\ref{SPPRT}) can be equivalently written as Problem (\ref{SPP}), and that therefore the MISOCP formulations presented in Section \ref{ss:22} are also valid for representing it. 
Let $\sigma : E \to \{m+1,...,m+|E|\}$ any bijective map between the set of edges $E$ and the set of natural numbers $\{m+1,...,m+|E|\}$. We can extend the set of polyhedra $\{P_1,..,P_m\}$ to $\{P_1,..,P_m,P_{m+1},...,P_{m+|E|}\}$ by defining, for each $h\in \{1,...,|E|\}$, $P_{m+h} = F_{ij}$ being $\sigma((i,j)) = m+h$. In the same way, $\ell_{p_{m+h}} = \ell_{p_{ij}}$ and  $\omega_{m+h} = \omega_{ij}$ being $\sigma((i,j)) = m+h$. In addition, for all $i\in \{1,...,m\}$ and $h\in\{1,...,|E|\}$ such that $P_{m+h}\subseteq P_i$ we denote $F_{i m+h}= P_i \cap P_{m+h} = P_{m+h}$ and $F_{m+h i} = F_{i m+h}$. 

At this point we can define the graph $G^{\text{RT}}=(V^{\text{RT}},E^{\text{RT}})$ such that $V^{\text{RT}}=\{1,...,m+|E|\}$ and $\{i,j\}\in A^{\text{RT}}$ iff $P_i\subseteq P_j$ or $P_j\subseteq P_i$, $i,j=1,...,m$, $i\not=j$. Now observe that Problem (\ref{SPPRT}) is equivalent to the following Problem 
(\ref{SPP}):
\begin{align*}
D(x_s,x_t) = \inf_{\scriptsize\begin{array}{c}(s,i_1,i_2,\cdots ,i_{k-1},i_k, t)\in \Gamma^{\text{RT}}  \\ y_{s i_1}\in F_{s i_1},\ldots,y_{i_kt}\in F_{i_kt}\end{array}} \omega_s\|x_s-y_{s i_1}\|_{p_{s}} + \omega_{i_1}\|y_{si_1}-y_{i_1 i_2}\|_{p_{i_1}} + \ldots + \nonumber\\
 \omega_{ i_k}\|y_{i_{k-1}i_k}-y_{i_k i_t}\|_{p_{i_k}} + \omega_{t}\|y_{i_k t}-x_t\|_{p_{t}}
\end{align*}
where $\Gamma^{\text{RT}}$ is the set of all the simple paths from $s$ to $t$ in $G^{\text{RT}}$.

\subsubsection{Non-simple paths}
\label{sss:242}

Another extension that can be considered concerns the number of visits allowed to each polyhedron. In Problem (\ref{SPP}), each polyhedron can only be visited once. However, allowing a path to visit a polyhedron more than once can lead to a decrease of the length of the shortest path. This situation is illustrated in Fig. 1 where the path that goes from $x_s$ to $x_t$ visits polyhedra $\{P_1,P_2,P_3\}$ in the order $P_3,P_2,P_1,P_2,P_1$, which is cheaper than \ML{the direct distance} between $x_s$ and $x_t$ in polyhedron $P_3$. \MRM{More examples of this behavior can be found
in Example 26 in \cite{refraction2017} \ML{(where its relation with the critical reflection physical phenomenon is also commented)}, in Fig. 1 in \cite{plastria2019I}, and for a case with more than two polyhedra in Figure 3.1 in \cite{mitchellWRP}}. 
\MRM{In the following we show that the MISOCP formulations given in Section \ref{ss:23} are valid for the shortest path problem in which each polyhedron can be visited at most twice, which we call Problem (SPP$^2$). As before, we reduce the problem to Problem (\ref{SPP}).}

\emph{MISOCP representation:} First, we extend the set of polyhedra $\{P_1,..,P_m\}$ to $\{P_1,..,P_m,P_{m+1},...,P_{2m}\}$ being $P_{m+i} = P_{i}$, $i=1,...,m$. Accordingly, we consider $F_{ij}=P_i \cap P_j$, $i,j=1,...,2m$, $i\not= j$ and $|i-j|\not= m$. The norm and the weight associated to polyhedron $P_{m+i}$ are $\ell_{m+i}=\ell_i$ and $\omega_{m+i}=\omega_i$, respectively, $i=1,...,m$. Now we can define the graph $G^{2}=(V^{2},E^{2})$ such that $V^{2}=\{1,...,2m\}$ and $(i,j)\in E^{2}$ iff $P_i\cap P_j\not= \emptyset$, $i,j=1,...,2m$, $i\not= j$ and $|i-j|\not= m$. \MRM{Finally note that Problem (\ref{SPP}) defined over the auxiliary graph $G^2$ is indeed Problem (SPP$^2$).} 

We note that by the ``no free lunch'' theorem, the price to pay for \MRM{this extension and the previous one} will be an increase in the number of (binary and continuous) variables of the formulations that will result in a higher computational time to solve the problems.

It is an interesting question to ask whether there exists a simple upper bound on the number of visits of a shortest path  to a generic traversed region. To answer , this question a rigorous analysis, similar to the one made in \cite{plastria2019I} for the case of two regions, should be carried out. However, this analysis is beyond of the scope of this paper.

\section{Location problems with demand points in a continuous framework with different norms}
\label{s:3}

In this section we study the problem of finding the location of a new facility with respect to a set of given demand points in such a way that the sum of the weighted distances between the new facility and the demand points is minimized. The difference with respect to the classical single-facility Weber problem \cite{NP05} is that distances are measured with the geodesic distance \MRM{induced by Problem (\ref{SPP}).}     

Let $\{P_1,...,P_m\}$ be a subdivision in polyhedra of $\mathbb{R}^d$ where each polyhedron $P_i$ is endowed with a different $\ell_{p_i}$-norm and has associated a weight $\omega_i>0$, $i=1,...,m$.  Consider the set of demand points $x_1,...,x_n\in \mathbb{R}^d$ with associated weights $w_1,...,w_n>0$. Let $N=\{1,...,n\}$ be the index set of the demand points and let $s_l$ be the index in $\{1,...,m\}$ such that $x_{l}\in P_{s_l}$, $l\in N$ (in case one of the points $x_1,...,x_n$ is on a face that belongs to two or more polyhedra, as we did in the previous section, we assume it belongs to exactly one of them which is determined in advance). In order to formally state the problem, we use the undirected graph $G=(V,E)$ such that $V=\{1,...,m\}$ and $\{i,j\}\in E$ iff $P_i$ and $P_j$ share a face, $i,j=1,...,m$, $i\not=j$. Then, the location problem can be stated as
\begin{align}
\inf_{t\in V, x\in P_t} \sum_{l\in N} w_l D_t (x_l, x), \label{LP} \tag{${\rm LocP}$}
\end{align}
where $D_t(x_l,x) = \omega_{s_l}\|x_l - x\|_{p_{s_l}}$ if $s_l = t$, \MRM{and} otherwise
\begin{align*}
D_t(x_l,x) = \inf_{\scriptsize\begin{array}{c}(s_l,i_1,i_2,\cdots ,i_{k-1},i_k, t)\in \Gamma_{s_l t}  \\ y_{s_l i_1}\in F_{s_l i_1},\ldots,y_{i_kt}\in F_{i_kt}\end{array}} \omega_{s_l}\|x_l-y_{s_l i_1}\|_{p_{s_l}} + \omega_{i_1}\|y_{s_l i_1}-y_{i_1 i_2}\|_{p_{i_1}} + \ldots + \nonumber\\
 \omega_{ i_k}\|y_{i_{k-1}i_k}-y_{i_k i_t}\|_{p_{i_k}} + \omega_{t}\|y_{i_k t}-x\|_{p_{t}}
\end{align*}
being $\Gamma_{s_l t}$ the set of all the simple paths from $s_l$ to $t$ in $G$. An illustrative instance of Problem \eqref{LP} is shown in Fig. \ref{instanceLP}.

\begin{figure}[H]
\centering
\includegraphics[scale=0.8]{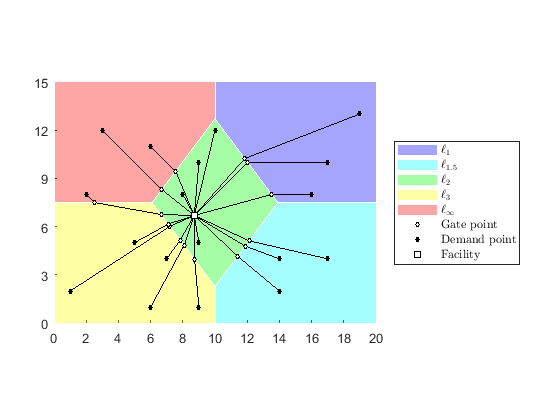}
\caption{An instance of Problem \eqref{LP}. }\label{instanceLP}
\end{figure}

By similar arguments as the ones given for Problem (\ref{SPP}), it is not difficult to see that Problem (\ref{LP}) is well-defined and \justo{one can reduce to polytopes the search for the facility and the gate points.} For this reason in the following w.l.o.g. we assume that $P_1,...,P_m$ are polytopes, which implies that their faces $F_{ij}$ are also polytopes. We also assume that polytopes $P_i, F_{ij}$ are encoded by \ML{their} set of extreme points $\text{Ext}(P_i), \text{Ext}(F_{ij})$. 

\MRM{Regarding the solvability of Problem (\ref{LP}), since Problem (\ref{SPP}) is not solvable in any algebraic computation model over the rational numbers, Problem (\ref{LP}) is not either.}

\subsection{MISOCP formulations}
\label{ss:31}

\MRM{We present two different mathematical programming formulations for Problem \eqref{LP} based on the ones given for Problem \eqref{SPP} in Section \ref{ss:22}. We start by the one based on formulation (SPP-F1). We recall that $G^{\text{D}}=(V,A)$ is the directed version of the undirected graph $G=(V,E)$, i.e. $A=\{(i,j),(j,i): \{i,j\}\in E\}$.}

\begin{theorem}
\label{theo6}
Problem \eqref{LP} is equivalent to the following problem, \MRM{called} \normalfont(LocP-F1):
\begin{subequations}
\label{PB6}
\begin{small}
\begin{align}
\min & \sum_{l \in N} \sum_{i\in V} w_l \omega_i  d_{i}^l\label{c:6a}\\
\mbox{s.t. }  & \sum_{i\in V} u_{i} = 1, \qquad \forall i\in V, \label{c:6b}\\
& \sum_{e\in \textnormal{Ext}(P_{i})} \mu_{ie} = u_i, \qquad \forall i\in V, \label{c:6c}\\
& \sum_{(s_l,j)\in A} z_{s_lj}^l - \sum_{(h,s_l)\in A} z_{hs_l}^l = 1 - u_{s_l}, \qquad \forall l\in N,\label{c:6d}\\
& \sum_{(i,j)\in A} z_{ij}^l - \sum_{(h,i)\in A} z_{hi}^l= -u_i, \qquad\forall  l\in N,i\in V\setminus \{s_l\},\label{c:6e}\\
& \sum_{(h,i)\in A} z_{hi}^l \leq  1-u_{s_l}, \qquad\forall l\in N,i\in V,\label{c:6f}\\
& \sum_{(i,j)\in A} z_{ij}^l \leq  1-u_{s_l}, \qquad\forall l\in N,i\in V,\label{c:6g}\\
& d_{s_l}^l \geq \| \sum_{e\in \textnormal{Ext}(P_{s_l})} \mu_{s_l e} e + \sum_{(s_l,j)\in A} \sum_{e\in \textnormal{Ext}(F_{s_lj})}\lambda_{s_lje}^l e - x_{l}\|_{p_{s_l}}, \qquad \forall  l\in N,\label{c:6h}\\
& d_{i}^l \geq \|\sum_{e\in \textnormal{Ext}(P_{i})} \mu_{ie} e + \sum_{(i,j)\in A} \sum_{e\in \textnormal{Ext}(F_{ij})}\lambda_{ije}^l e - \sum_{(h,i)\in A} \sum_{e\in \textnormal{Ext}(F_{hi})}\lambda_{hie}^l e\|_{p_i}, \qquad\forall  l\in N,i\in V\setminus\{s_l\},\label{c:6i}\\
& \sum_{e\in \textnormal{Ext}(F_{ij})}\lambda_{ije}^{l} = z_{ij}^l, \qquad\forall l\in N,(i,j)\in A,\label{c:6j}\\
& d_{i}^l\geq 0, \qquad\forall l\in N,i\in V, \label{c:6k}\\
& z_{ij}^l\in\{0,1\}, \qquad\forall l\in N,(i,j)\in A,\label{c:6l}\\
& \lambda_{ije}^l\geq 0, \qquad\forall l\in N,(i,j)\in A,e\in\textnormal{Ext}(F_{ij}),\label{c:6m}\\
& \mu_{ie}\geq 0, \qquad \forall i\in V, e\in\textnormal{Ext}(P_{i}),\label{c:6n}\\
& u_i\in \{0,1\}, \qquad\forall i\in V. \label{c:6o}
\end{align}
\end{small}
\end{subequations}
\end{theorem}

\begin{proof}
The domain of the decision variables in Problem (LocP-F1) is stated in \eqref{c:6k}-\eqref{c:6o}.

Variable $u_i\in \{0,1\}$ indicates whether the facility is located in polytope $P_i$ or not, $i\in V$. Then, constraint \eqref{c:6b} ensures that exactly one polytope of the set $\{P_1,...,P_m\}$ is chosen for placing the facility. Assuming $P_i$ is the selected polytope, variables $\mu_{ie}$, $e\in \text{Ext}(P_i)$, are used to represent $x$ as the convex combination of the extreme points of $P_{i}$, i.e., $x=\sum_{\text{Ext}(P_i)}\mu_{ie} e$. This representation is stated in constraints \eqref{c:6c}, which also set to zero the scalars associated to the extreme points of the non-selected polytopes. 

Variables $z_{ij}^l,\lambda_{ije}^l$ are similar to variables $z_{ij},\lambda_{ije}$ in formulation (SPP-F1) but replicated $n$ times, one for each $l\in N$, to compute the shortest path between $x_l$ and $x$. Then, if we assume that $u_t=1$ for some $t\in V$, note that constraints \eqref{c:6d}-\eqref{c:6g} force variables $z_{ij}^l$ to codify a simple path between $s_l$ and $t$ if $s_l\neq t$. If  $s_l= t$, constraints \eqref{c:6d}-\eqref{c:6g} set variables $z_{ij}^l$ to zero. Similarly to constraints \eqref{c:2j} in formulation (SPP-F1), constraints \eqref{c:6j} determine by means of variables $\lambda_{ije}^l$ the gate points in the traversed intermediate faces $F_{ij}$, and set to zero these variables for the non-traversed ones. 

\MRM{Taking into account the above and the rationale of formulation (SPP-F1) for representing Problem \eqref{SPP}, the equivalence between Problem (LocP-F1) and Problem \eqref{LP} follows.} \qed
\end{proof}

\MRM{The formulation for Problem  \eqref{LP} based on formulation (SPP-F2) is shown in Theorem \ref{theo7}.}

\begin{theorem}
\label{theo7}
Problem \eqref{LP} is equivalent to the following problem, \MRM{called} \normalfont(LocP-F2):
\justo{\begin{subequations}
\label{PB7}
\begin{align}
\min & \sum_{l\in N}\sum_{i\in V} w_l \omega_i d_{i}^l\label{c:7a}\\
\mbox{s.t. } & \sum_{i\in V} u_{i} = 1, \label{c:7b}\\
& \sum_{e\in \textnormal{Ext}(P_{i})} \mu_{ie} = u_i, \qquad \forall i\in V, \label{c:7c}\\
& \sum_{(s_l,j)\in A} z_{s_lj}^l - \sum_{(h,s_l)\in A} z_{hs_l}^l = 1 - u_{s_l}, \qquad \forall l\in N,\label{c:7d}\\
& \sum_{(i,j)\in A} z_{ij}^l - \sum_{(h,i)\in A} z_{hi}^l= -u_i, \qquad \forall  l\in N,i\in V\setminus \{s_l\},\label{c:7e}\\
& \sum_{(h,i)\in A} z_{hi}^l \leq  1-u_{s_l}, \qquad \forall l\in N,i\in V,\label{c:7f}\\
& \sum_{(i,j)\in A} z_{ij}^l \leq  1-u_{s_l}, \qquad \forall l\in N,i\in V,\label{c:7g}\\
& d_{s_l}^l \geq \sum_{(s_l,j)\in A}\| \sum_{e\in \textnormal{Ext}(F_{s_l j})}\lambda_{s_lje}^l e - x_l z_{s_lj}^l  \|_{p_{s_l}} \nonumber\\
&  + \|\sum_{e\in \textnormal{Ext}(P_{s_l})} \mu_{s_l e} e- x_l u_{s_l}\|_{p_{s_l}}, \qquad     \forall l\in N, \label{c:7h}\\
& d_{i}^l \geq \sum_{(h,i)\in A}\sum_{(i,j)\in A}\|\sum_{e\in \textnormal{Ext}(F_{ij})}\Psi_{hije}^l e - \sum_{e\in \textnormal{Ext}(F_{hi})}\Phi_{hije}^{l} e\|_{p_i}, \qquad  \nonumber\\
& + \sum_{(h,i)\in A}\|\sum_{e\in \textnormal{Ext}(P_i)}\Theta_{hie}^l e - \sum_{e\in \textnormal{Ext}(F_{hi})}\Upsilon_{hie}^{l} e\|_{p_i}, \qquad \forall l\in N, i\in V\setminus\{s_l\},\label{c:7i}\\
& \sum_{e\in \textnormal{Ext}(F_{ij})}\lambda_{ije}^l = z_{ij}^l, \qquad \forall l\in N, (i,j)\in A,\label{c:7j}\\
&  \sum_{(i,j)\in A} \rho_{hij}^l + \sum_{e \in \text{Ext}(P_{i})}
\Theta_{hie}^{l}  =  z_{hi}^l, \qquad \forall l\in N, i\in V\setminus\{s_l\},(h,i)\in A,\label{c:7k}\\
&  \sum_{(h,i)\in A} \rho_{hij}^l =  z_{ij}^l, \qquad \forall l\in N, i\in V\setminus\{s_l\},(i,j)\in A,\label{c:7l}\\
&  \sum_{e\in \text{Ext}(F_{hi})} \Phi^{l}_{hije} =  \rho_{hij}^l, \qquad \forall l\in N, i\in V\setminus\{s_l\},(h,i),(i,j)\in A,\label{c:7m}\\
&  \sum_{(i,j)\in A} \Phi^{l}_{hije}+\Upsilon_{hie}^{l} =  \lambda_{hie}^l, \qquad \forall l\in N, i\in V\setminus\{s_l\},(h,i)\in A, e \in \text{Ext}(F_{hi}),\label{c:7n}\\
&  \sum_{e\in \text{Ext}(F_{ij})} \Psi^{l}_{hije} =  \rho_{hij}^l, \qquad \forall l\in N, i\in V\setminus\{s_l\},(h,i),(i,j)\in A,\label{c:7o}\\
&  \sum_{(h,i)\in A} \Psi^{l}_{hije} =  \lambda_{ije}^l, \qquad \forall l\in N, i\in V\setminus\{s_l\},(i,j)\in A, e \in \text{Ext}(F_{ij}),\label{c:7p}\\
&  \sum_{(h,i)\in A} \sum_{e \in \text{Ext}(F_{hi})}
\Upsilon_{hie}^{l} = u_i, \qquad \forall l\in N, i\in V\setminus\{s_l\},\label{c:7q}\\
&  \sum_{(h,i)\in A} \sum_{e \in \text{Ext}(P_{i})}
\Theta_{hie}^{l} = u_i, \qquad \forall l\in N, i\in V\setminus\{s_l\},\label{c:7r}\\
&   \sum_{(h,i)\in A} \Theta_{hie}^{l} = \mu_{ie}, \qquad \forall l\in N, i\in V\setminus\{s_l\},e\in \text{Ext}(P_i),\label{c:7s}\\
& u_i\in \{0,1\}, \qquad \forall i\in V,\label{c:7t}\\
& \mu_{ie}\geq 0, \qquad \forall i\in V, e \in \textnormal{Ext}(P_i),\label{c:7u}\\
& d_{i}^l\geq 0, \qquad \forall l\in N,i\in V,\label{c:7v}\\
& z_{ij}^l\in\{0,1\},\qquad \forall l\in N,(i,j)\in A,\label{c:7w}\\
& \lambda_{ije}^l\geq 0, \qquad \forall l\in N,(i,j)\in A, e\in\textnormal{Ext}(F_{ij}),\label{c:7x}\\
& \rho_{hij}^l\geq 0, \qquad \forall l\in N,i\in V\setminus\{s_l\},(h,i),(i,j)\in A,\label{c:7y}\\
& \Phi_{hije}^{l}\geq 0, \qquad \forall l\in N,i\in V\setminus\{s_l\},(h,i),(i,j)\in A, e\in \textnormal{Ext}(F_{hi}),\label{c:7z}\\
& \Psi_{hije}^{l}\geq 0, \qquad \forall l\in N,i\in V\setminus\{s_l\},(h,i),(i,j)\in A, e\in \textnormal{Ext}(F_{ij}),\label{c:7aa}\\
& \Upsilon_{hie}^{l}\geq 0, \qquad \forall l\in N,i\in V\setminus\{s_l\},(h,i)\in A, e\in \textnormal{Ext}(F_{hi}),\label{c:7ab}\\
& \Theta_{hie}^{l}\geq 0, \qquad \forall l\in N,i\in V\setminus\{s_l\},(h,i)\in A, e\in \textnormal{Ext}(P_{i}).\label{c:7ac}
\end{align}
\end{subequations}}
\end{theorem}

\begin{proof}
\MRM{Constraints \eqref{c:6b}-\eqref{c:6g},\eqref{c:6j},\eqref{c:6k}-\eqref{c:6o} in Problem (LocP-F1) are constraints \eqref{c:7b}-\eqref{c:7g},\eqref{c:7j}, \eqref{c:7t}-\eqref{c:7x} in Problem (LocP-F2). Thus, \ML{the statement} in the proof of Theorem \ref{theo6} for these constraints is also valid here.} 

Taking into account the above, observe that Problem \eqref{LP} can be equivalently written as
\begin{align}
\min &  \sum_{l\in N}\sum_{i\in V} w_l \omega_i d_{i}^l \label{LocAuxProblem}  \\
\mbox{s.t. }& \eqref{c:7b} \text{-} \eqref{c:7g},\nonumber\\
& d_{s_l}^l \geq \sum_{(s_l,j)\in A}\| \sum_{e\in \textnormal{Ext}(F_{s_l j})}\lambda_{s_lje}^l e - x_lz_{sj}  \|_{p_{s_l}} & \nonumber\\
&  + \|\sum_{e\in \textnormal{Ext}(P_{s_l})} \mu_{s_l e} e- x_l u_{s_l}\|_{p_{s_l}} , & \forall l\in N, \nonumber\\
& d_{i}^l \geq \sum_{(h,i)\in A}\sum_{(i,j)\in A}\|\sum_{e\in \textnormal{Ext}(F_{ij})}\lambda_{ije}^l e - \sum_{e\in \textnormal{Ext}(F_{hi})}\lambda_{hie}^{l} e\|_{p_i}z_{hi}z_{ij}(1-u_i)\hspace{-1cm} & \nonumber\\
& + \sum_{(h,i)\in A}\|\sum_{e\in \textnormal{Ext}(P_{s_l})} \mu_{s_l e} e- \sum_{e\in \textnormal{Ext}(F_{hi})}\lambda_{hie}^{l} e\|_{p_i} z_{hi}u_i,\hspace{-1cm} &\forall l\in N, i\in V\setminus\{s_l\},\nonumber\\
& \eqref{c:7j}, \eqref{c:7t} \text{-} \eqref{c:7x}. \nonumber 
\end{align}

Let us see that Problem (Loc-F2) is equivalent to Problem \eqref{LocAuxProblem}. Assume $u_{t}=1$ for some $t\in V$ and let $l'\in N$. First, suppose $s_{l'}=t$. Then, since constraints \eqref{c:7d}-\eqref{c:7g} set variables $z_{ij}^{l'}$ to zero,  it is straightforward to see that constraints \eqref{c:7j}-\eqref{c:7s} set to zero all variables $\lambda_{ije}^{l'},  \rho_{hij}^{l'},  \Phi_{hije}^{l'} ,\Psi_{hije}^{l'},\Upsilon_{hie}^{l'},\Theta_{hie}^{l'}$. Now, suppose $s_{l'}\neq t$. Then, on the one hand, constraints \eqref{c:7j}-\eqref{c:7s} impose 
\begin{subequations}
\begin{align}
&  \sum_{(i,j)\in A} \rho_{hij}^{l'}  =  z_{hi}^{l'}, &\forall i\in V\setminus\{s_{l'},t\},(h,i)\in A,\label{c:8k}\\
&  \sum_{(h,i)\in A} \rho_{hij}^{l'} =  z_{ij}^{l'}, &\forall i\in V\setminus\{s_{l'},t\},(i,j)\in A,\label{c:8l}\\
&  \sum_{e\in \text{Ext}(F_{hi})} \Phi^{l'}_{hije} =  \rho_{hij}^{l'}, &\forall i\in V\setminus\{s_{l'},t\},(h,i),(i,j)\in A,\label{c:8m}\\
&  \sum_{(i,j)\in A} \Phi^{l'}_{hije} =  \lambda_{hie}^{l'}, &\forall i\in V\setminus\{s_{l'},t\},(h,i)\in A, e \in \text{Ext}(F_{hi}),\label{c:8n}\\
&  \sum_{e\in \text{Ext}(F_{ij})} \Psi^{l'}_{hije} =  \rho_{hij}^{l'}, &\forall i\in V\setminus\{s_{l'},t\},(h,i),(i,j)\in A,\label{c:8o}\\
&  \sum_{(h,i)\in A} \Psi^{l'}_{hije} =  \lambda_{ije}^{l'}, &\forall i\in V\setminus\{s_{l'},t\},(i,j)\in A, e \in \text{Ext}(F_{ij}),\label{c:8p}
\end{align}
\end{subequations} 
and set to zero all variables $\Upsilon_{hie}^{l'},\Theta_{hie}^{l'}$ for $i\in V\setminus\{s_{l'},t\}$.
On the other hand, constraints \eqref{c:7j}-\eqref{c:7s} impose 
\begin{subequations}
\begin{align}
& \sum_{e \in \text{Ext}(F_{ht})}
\Theta_{hte}^{l'}  =  z_{ht}^{l'}, &\forall (h,t)\in A,\label{c:9j}\\
&  \Upsilon_{hte}^{l'} =  \lambda_{hte}^{l'}, &\forall (h,t)\in A, e \in \text{Ext}(F_{ht}),\label{c:9m}\\
&  \sum_{(h,t)\in A} \sum_{e \in \text{Ext}(F_{ht})}
\Upsilon_{hte}^{l'} = 1, &\forall (h,t)\in A,\label{c:9q}\\
&  \sum_{(h,t)\in A} \sum_{e \in \text{Ext}(F_{ht})}
\Theta_{hte}^{l'} = 1, &\forall (h,t)\in A,\label{c:9r}\\
&  \sum_{(h,t)\in A} \Theta_{hte}^{l'} = \mu_{te}, &\forall e\in \text{Ext}(P_t),\label{c:9s}
\end{align}
\end{subequations}
and set to zero all variables $\Phi_{htje}^{l'} ,\Psi_{htje}^{l'}$. Thus, by similar arguments to the ones made in the proof of Theorem \eqref{theo3}, it can be shown that the values of variables $\rho_{hij}^{l},\Phi_{hije}^l, \Psi_{hije}^l,\Theta_{hie}^l,\Upsilon_{hie}^l$ are set to 
the products $z_{hi}^lz_{ij}^l,\lambda_{hie}^l\rho_{hij}^l(1-u_i),\lambda_{ije}^l\rho_{hij}^l(1-u_i),\mu_{ie} z_{hi}^l u_i$ and $\lambda_{hie}^l z_{hi}^l u_i$, respectively. Finally, the equivalence follows by applying the homogeneity property of the norms in Problem \eqref{LocAuxProblem}.\qed
\end{proof}

\MRM{Formulation (LocP-F1) and formulation (LocP-F2) can be rewritten as MISOCPs using the reformulation technique described in Section \ref{ss:22}.}

\MRM{We establish below a result \justo{comparing} formulations (LocP-F1) and (LocP-F2) similar to the one in Proposition \label{propRelaxationLoc} for formulations (SPP-F1) and (SPP-F2).}

\begin{prop}
\label{propRelaxationLoc}
Let $\zeta,\zeta'\geq 0$ be the objective values of the continuous relaxations of formulation (LocP-F1) and formulation (LocP-F2), respectively. Then, $\zeta\leq\zeta'$. 
\end{prop}

\begin{proof}
Consider the relaxations of Problem (LocP-F1) and Problem (LocP-F2). Note that the set of decision variables of the first problem is contained in the set of decision variables of the second one. Note also that the feasible domain for the set of variables $z_{ij}^l, \lambda_{ije}^l,\mu_{ie}$ and $u_i$ in the first problem contains the corresponding feasible domain in the second problem.
Let us assume that we are given a realization for this set of variables that is feasible in both problems. 

\MRM{Let $l\in N$}. Then, since contraints \eqref{c:6d} and \eqref{c:6g} imply $ \sum_{(s_l,j)\in A} z_{s_lj}^l = 1-u_{s_l}$, and using  the triangular inequality, we have that
\begin{subequations}
\begin{align}
& \| \sum_{e\in \textnormal{Ext}(P_{s_l})} \mu_{s_l e} e + \sum_{(s_l,j)\in A} \sum_{e\in \textnormal{Ext}(F_{s_lj})}\lambda_{s_lje}^l e - x_{l}\|_{p_{s_l}}\nonumber\\
&  = \| \sum_{e\in \textnormal{Ext}(P_{s_l})} \mu_{s_l e} e + \sum_{(s_l,j)\in A} \sum_{e\in \textnormal{Ext}(F_{s_lj})}\lambda_{s_lje}^l e - x_{l}(\sum_{(s_l,j)\in A} z_{s_lj}^l+u_{s_l})\|_{p_{s_l}}\nonumber\\
& = \| \sum_{(s_l,j)\in A} (\sum_{e\in \textnormal{Ext}(F_{s_lj})}\lambda_{s_lje}^l e - x_{l}z_{s_lj}^l)+(\sum_{e\in \textnormal{Ext}(P_{s_l})} \mu_{s_l e} e-x_l u_{s_l})\|_{p_{s_l}} \nonumber\\
& \leq  \sum_{(s_l,j)\in A}\| \sum_{e\in \textnormal{Ext}(F_{s_l j})}\lambda_{s_lje}^l e - x_l z_{s_lj}  \|_{p_{s_l}}+ \|\sum_{e\in \textnormal{Ext}(P_{s_l})} \mu_{s_l e} e- x_l u_{s_l}\|_{p_{s_l}}.\nonumber
\end{align}
\end{subequations}

On the other hand, by constraints \eqref{c:7n}, \eqref{c:7p} and \eqref{c:7s}, and by the triangular inequality, for each $i\in V\setminus\{s,t\}$ we have that
\begin{subequations}
\begin{align}
& \|\sum_{e\in \textnormal{Ext}(P_{i})} \mu_{ie} e + \sum_{(i,j)\in A} \sum_{e\in \textnormal{Ext}(F_{ij})}\lambda_{ije}^l e - \sum_{(h,i)\in A} \sum_{e\in \textnormal{Ext}(F_{hi})}\lambda_{hie}^l e\|_{p_i}\nonumber\\
&  = \|\sum_{e\in \textnormal{Ext}(P_{i})} \sum_{(h,i)\in A} \Theta_{hie}^{l} e + \sum_{(i,j)\in A} \sum_{e\in \textnormal{Ext}(F_{ij})} \sum_{(h,i)\in A} \Psi^{l}_{hije}  e - \sum_{(h,i)\in A} \sum_{e\in \textnormal{Ext}(F_{hi})}(\sum_{(i,j)\in A} \Phi^{l}_{hije}+\Upsilon_{hie}^{l}) e\|_{p_i}\nonumber\\
& = \|\sum_{(h,i)\in A}\sum_{(i,j)\in A} (\sum_{e\in \textnormal{Ext}(F_{ij})}\Psi_{hije}^l e - \sum_{e\in \textnormal{Ext}(F_{hi})}\Phi_{hije}^{l} e) + \sum_{(h,i)\in A}(\sum_{e\in \textnormal{Ext}(P_i)}\Theta_{hie}^l e - \sum_{e\in \textnormal{Ext}(F_{hi})}\Upsilon_{hie}^{l} e)\|_{p_i}\nonumber\\
& \leq  \sum_{(h,i)\in A}\sum_{(i,j)\in A}\|\sum_{e\in \textnormal{Ext}(F_{ij})}\Psi_{hije}^l e - \sum_{e\in \textnormal{Ext}(F_{hi})}\Phi_{hije}^{l} e\|_{p_i} + \sum_{(h,i)\in A}\|\sum_{e\in \textnormal{Ext}(P_i)}\Theta_{hie}^l e - \sum_{e\in \textnormal{Ext}(F_{hi})}\Upsilon_{hie}^{l} e\|_{p_i}.\nonumber
\end{align}
\end{subequations}

Finally we note that the above inequalities are preserved  
when the MISOCP representation of the norms is applied.
\qed
\end{proof}

\MRM{As done for Problem \eqref{SPP}, we also present a preprocessing scheme for Problem \eqref{LP}. After that, we report a computational study to evaluate the performance of the proposed MISOCP formulations and the preprocessing strategy.}

\subsubsection{Preprocessing}
\label{sss:311}

\MRM{As it happened with formulation (SPP-F2) in the case of Problem \eqref{SPP}, the continuous relaxation of formulation (LocP-F2) provides \justo{tighter} lower bounds for Problem \eqref{LP}. Thus, we can design a preprossecing algorithm for Problem \eqref{LP} inspired in Algorithm \ref{alg:preprocess} for Problem \eqref{SPP}. We use this preprocessing to discard some regions in which placing the new facility is not optimal. 
Let $\zeta_i$ be the objective value of the continuous relaxation of formulation (LocP-F2) but imposing that the new facility is located in polyhedron $P_i$. If $\zeta_i$ is greater than the objective value $\zeta^{\text{MIP}}$ of a feasible solution for Problem \eqref{LP}, then placing the new facility in polyhedron $P_i$ is not optimal and therefore all the variables associated \justo{to this decision} in a formulation for Problem \eqref{LP} can be removed. We preprocess a number $m^*$ of the  $m$ regions, since preprocessing all of them can be computationally expensive.
We decide to process first the polyhedra $P_i$ with a higher value of the measure $\mathcal{F}(P_i) = \min_{x\in P_i} \sum_{l\in N} \|x_l - x\|_2$. The upper bound on the objective function $\zeta^{\text{MIP}}$  that we consider is the objective value of Problem \eqref{LP} fixing for each demand point $x_l$ the simple path induced by the intersection of the segment $[x_l,x^*]$ with the set of polyhedra $\{P_1,...,P_m\}$, where $x^*$ is the optimal solution of the problem $\min_{x\in \mathbb{R}^d} \sum_{l\in N} \|x_l - x\|_2$. The basic pseudocode of the above preprocessing is shown \justo{in Algorithm \ref{alg:preprocessLoc}}. The preprocessing algorithm for Problem \eqref{LP} returns a list $\mathcal{L}$ of polyhedra in which placing the new facility is not optimal.}

\begin{algorithm}
\caption{\MRM{Preprocessing for Problem \eqref{LP}}}
\label{alg:preprocessLoc}
\renewcommand{\thealgorithm}{}
\floatname{algorithm}{}
\begin{algorithmic}[1]
    \Begin
    \State compute the optimal solution $x^*$ of the problem $\min_{x\in \mathbb{R}^d} \sum_{l\in N} \|x_l - x\|_2$
    \State compute the objective value $\zeta^{\text{MIP}}$ of Problem \eqref{LP} fixing for each demand point $x_l$ the simple path induced by the intersection of the segment $[x_l,x^*]$ with the set of polyhedra $\{P_1,...,P_m\}$ 
    \State sort polyhedra $P_{1},...,P_{m}$ in non-decreasing sequence with respect to $\mathcal{F}$: $P_{(1)},...,P_{(m)}$ such that $\mathcal{F}(P_{(1)})\geq...\geq \mathcal{F}(P_{(m)})$
    \State let $\mathcal{L}:=\emptyset$
    \For{$i=1,...,m^*$}
    \State compute $\zeta_{(i)}$ the objective value of the continuous relaxation of formulation (LocP-F2) but imposing that the new facility is located in polyhedron $P_i$
    \If{$\zeta_{(i)}>\zeta^{\text{MIP}}$}
    \State  update $\mathcal{L} := \mathcal{L}\cup \{P_{(i)}\}$
    \EndIf
    \EndFor
    \State \Return $\mathcal{L}$
    \End
\end{algorithmic}

\end{algorithm}

\subsection{Computational experiments}
\label{ss:32}

\MRM{We evaluate the performance of the MISOCP formulations and the preprocessing strategy for Problem \eqref{LP} following a similar scheme to the one described in Section \ref{ss:24}. The details about the software and the computer that we use in the  computational study have been already commented in Section \ref{ss:24}.}

We generate instances for Problem \eqref{LP} from data sets considered in the literature. Following the computational experience in \cite{refraction2017}, we have considered the demand points on the plane from the data sets in \cite{eilon-watson,Parlar1994,bsss} (see Fig. \ref{demandPoints}). Given one of those demand points data sets, we produce a polyhedral subdivision in a rectangle containing the demand points following the procedure described in Section \ref{ss:24}. In all generated instances, weights associated to polyhedra and demand points are set to 1, except in the instances generated from the data sets in \cite{bsss} for which we keep the weights associated to the demand points in the original paper.
As in Section \ref{ss:24} we set a time limit for each instance of 7200 seconds and we report for each value of $m$ the average, minimum and maximum value for CPU(s) and Gap(\%) over 5 instances.
The results obtained for formulations (LocP-F1) and (LocP-F2), as well as their combinations with the preprocessing in Algorithm \ref{alg:preprocess} (referred as ``Pre. + (LocP-F1)'' and ``Pre. + (LocP-F2)''), are shown in Tables \ref{tableL1}-\ref{tableL5}. In this case we decide to preprocess all the polytopes in the subdivision (we choose $m^*=m$ in Algorithm \ref{alg:preprocessLoc}). As in Section \ref{ss:24}, the preprocessing time is subtracted to the 7200 seconds assigned to the formulation for solving the instance.

\begin{figure}[H] 
  \subfloat[The 4-points data set in \cite{Parlar1994} within the square $\text{[}0,12\text{]}^2$.]{
	\begin{minipage}[c][1\width]{
	   0.45\textwidth}
	   \centering
	   \includegraphics[width=1.15\textwidth]{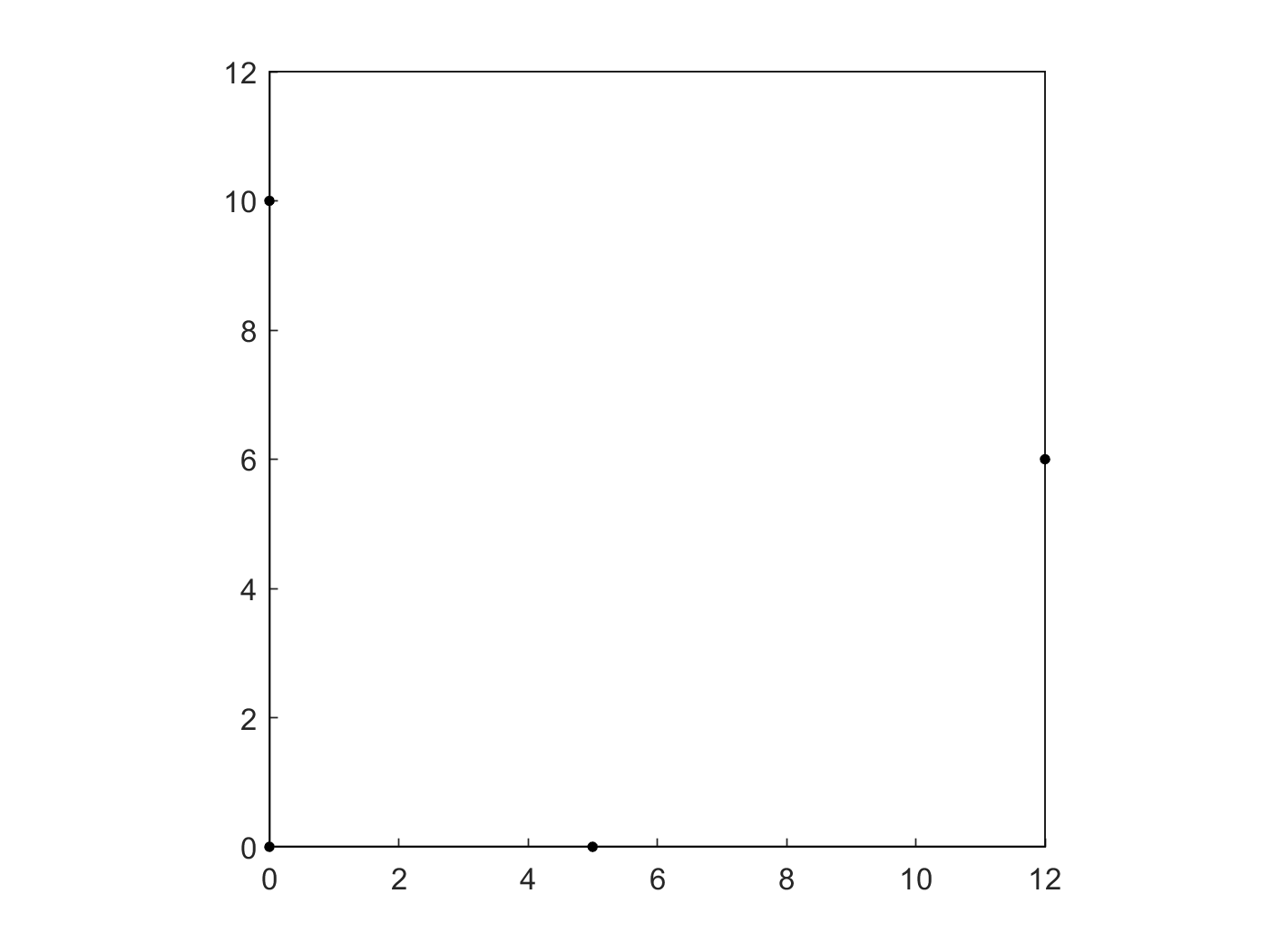}
	\end{minipage}}
 \hfill 	
  \subfloat[The 18-points data set in \cite{Parlar1994} within the rectangle $\text{[}0,20\text{]$\times$[}0,15\text{]}$.]{
	\begin{minipage}[c][1\width]{
	   0.45\textwidth}
	   \centering
	   \includegraphics[width=1.15\textwidth]{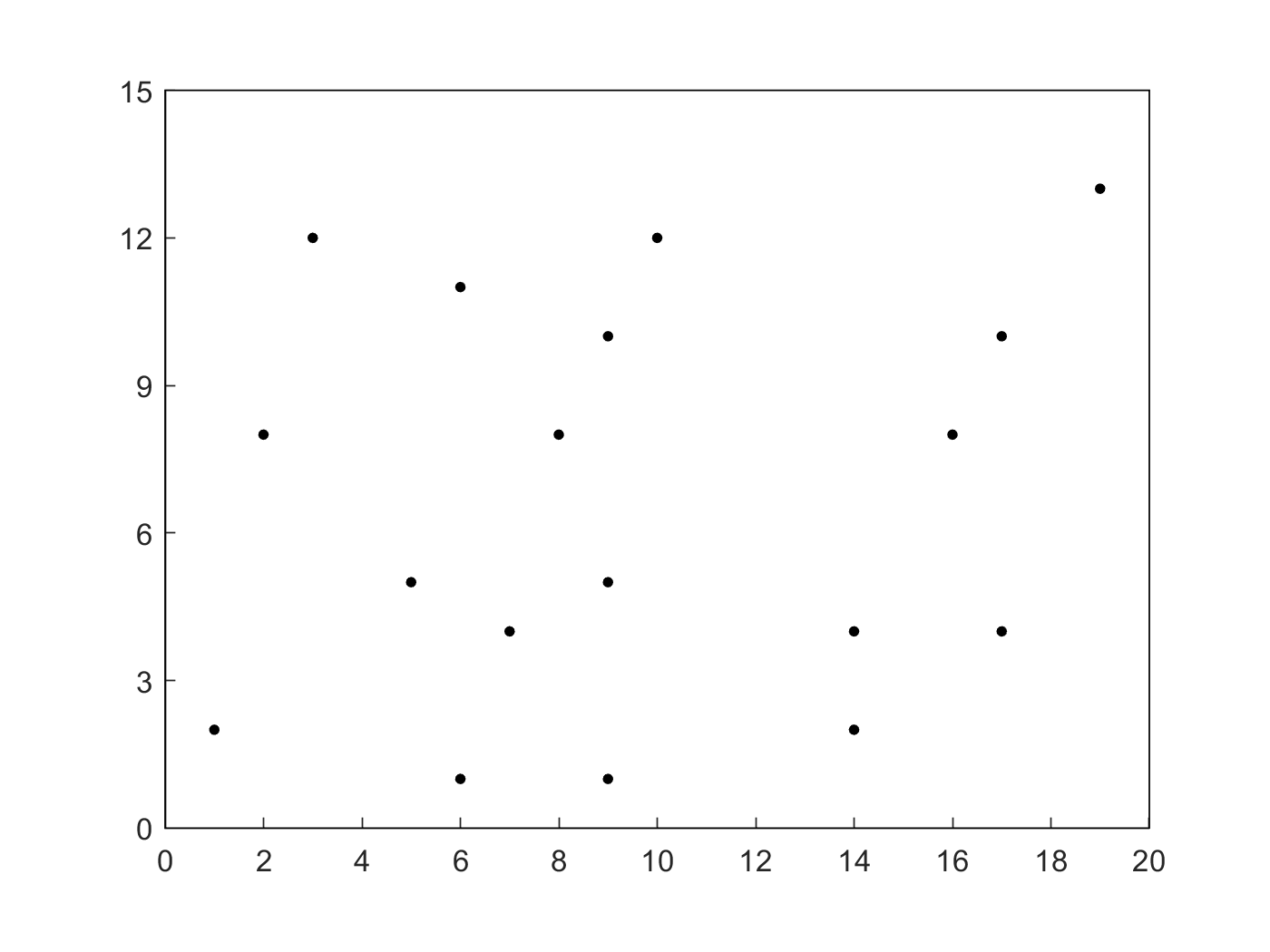}
	\end{minipage}}
\newpage
\hfill
  \subfloat[The 30-points data set in \cite{bsss} within the square $\text{[}0,10\text{]}^2$.]{
	\begin{minipage}[c][1\width]{
	   0.45\textwidth}
	   \centering
	   \includegraphics[width=1.15\textwidth]{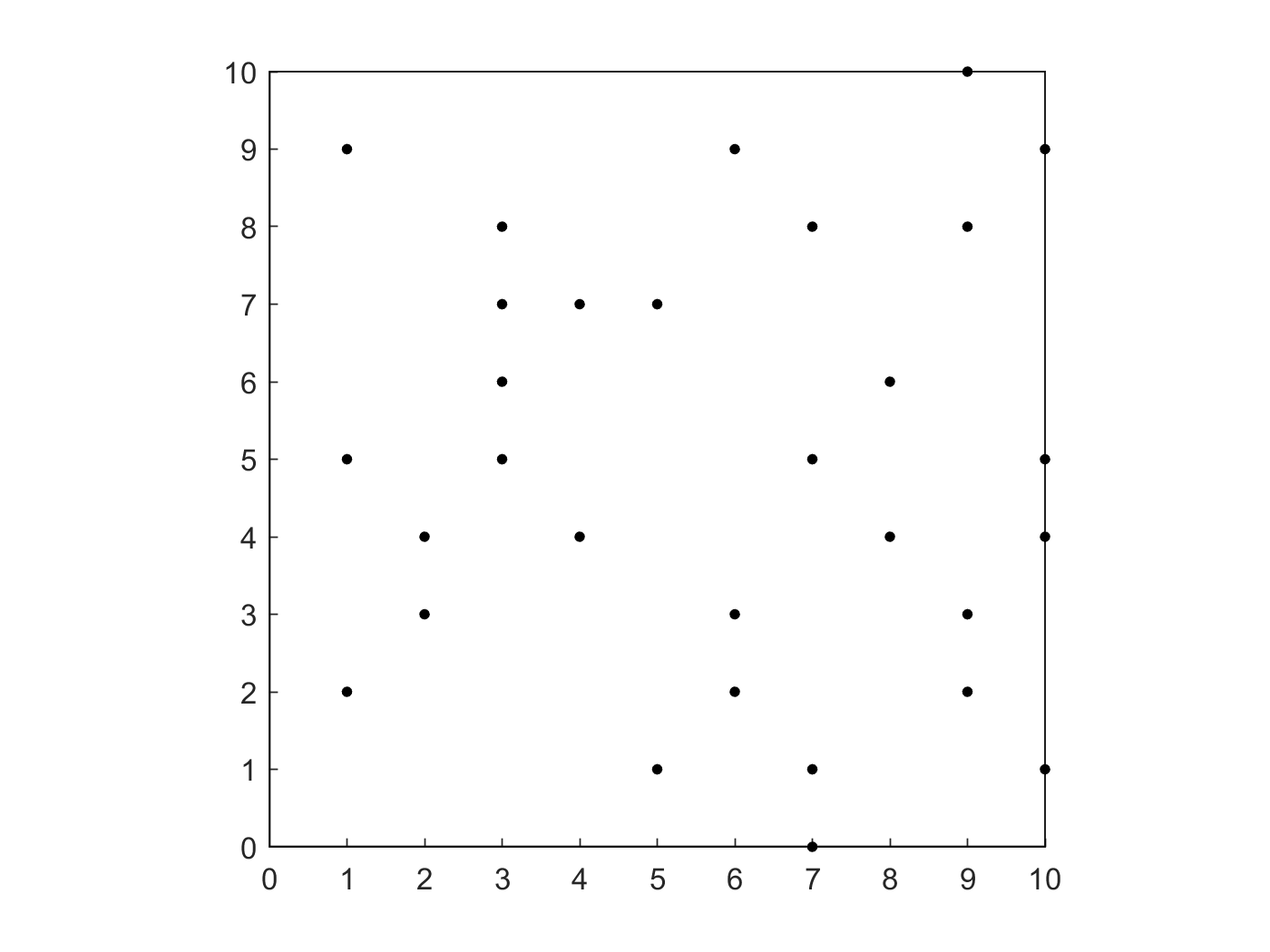}
	\end{minipage}}
 \hfill	
  \subfloat[The 50-points data set in \cite{bsss} within the square $\text{[}0,20\text{]}^2$.]{
	\begin{minipage}[c][1\width]{
	   0.45\textwidth}
	   \centering
	   \includegraphics[width=1.15\textwidth]{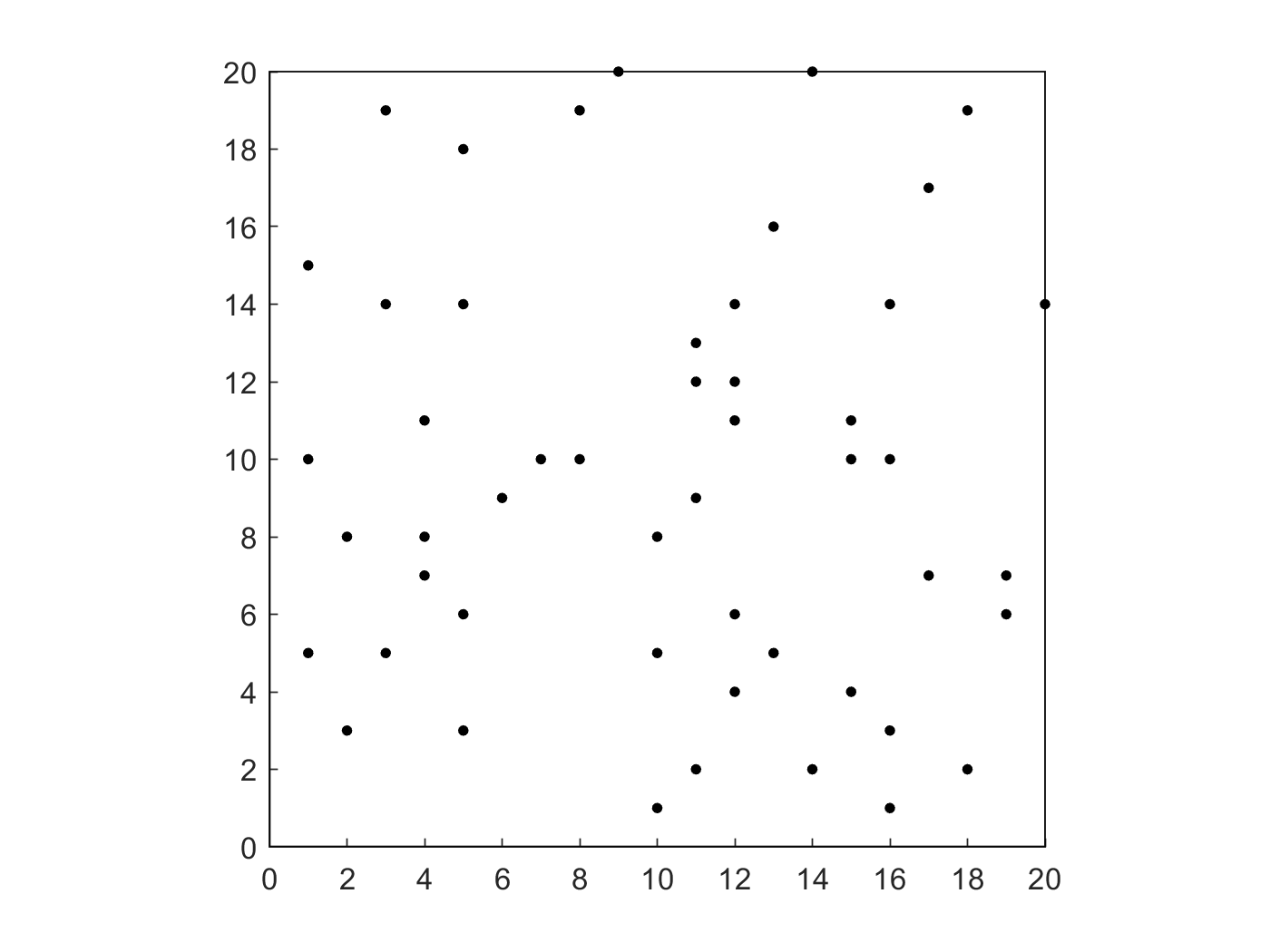}
	\end{minipage}}
	\hfill
  \subfloat[The 50-points data set in \cite{eilon-watson} within the square $\text{[}0,10\text{]}^2$.]{
	\begin{minipage}[c][1\width]{
	   0.45\textwidth}
	   \centering
	   \includegraphics[width=1.15\textwidth]{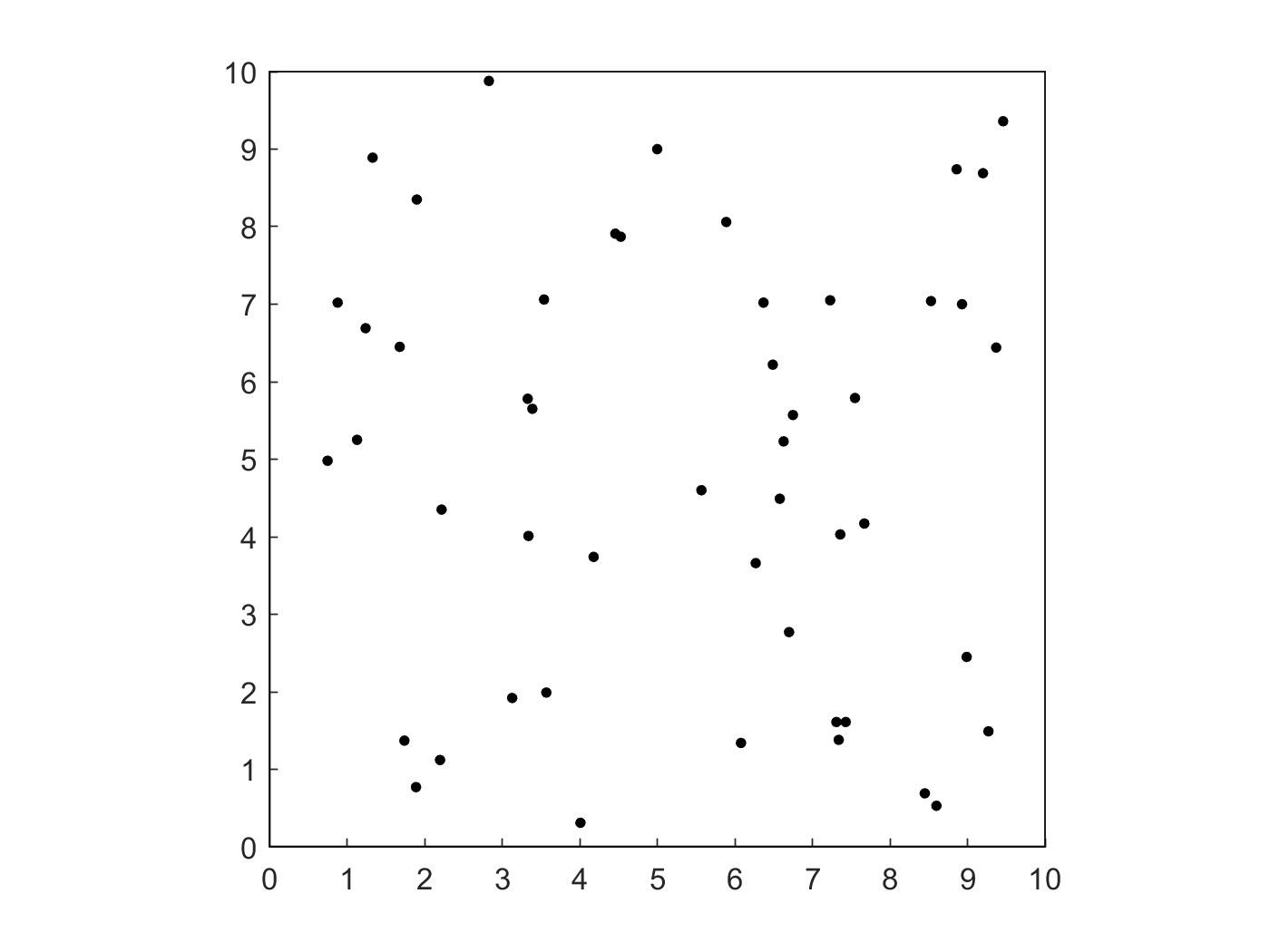}
	\end{minipage}}
\caption{Demand points from the data sets in \cite{eilon-watson,Parlar1994,bsss} within a rectangle.}\label{demandPoints}
\end{figure}

\begin{table}[H]
\begin{center}
{\tiny
\begin{tabular}{|r|rrr|rrr|r|rrr|rrr|}
\hline
(LocP-F1) &  \multicolumn{3}{c|}{CPU(s)} & \multicolumn{3}{c|}{gap(\%)} & 
(LocP-F2) &  \multicolumn{3}{c|}{CPU(s)} & \multicolumn{3}{c|}{gap(\%)}\\ \hline
 $m$        &	aver	&	min	&	max	&	aver	 & min	&	max	&	  
 $m$        &	aver	&	min	&	max	&	aver	 & min	&	max\\ \hline
5 &  0.92 &  0.67 &  1.25 &  0.00 &  0.00 &  0.00  &
5 &  2.17 &  1.11 &  2.81 &  0.00 &  0.00 &  0.01\\
10 &  3.88 &  3.46 &  4.49 &  0.00 &  0.00 &  0.01 &
10 &  11.84 &  8.85 &  15.31 &  0.00 &  0.00 &  0.01  \\
15 &  21.74 &  16.29 &  34.24 &  0.00 &  0.00 &  0.00 &
15 &  1490.12 &  30.60 &  7199.43 &  20.00 &  0.00 &  100.00\\
20 &  173.73 &  78.67 &  405.38 &  0.01 &  0.00 &  0.01 &
20 &  1886.48 &  72.14 &  7200.00 &  20.00 &  0.00 &  100.00 \\
30 &  3103.22 &  146.77 &  7200.00 &  0.03 &  0.01 &  0.07&
30 &  1341.46 &  259.14 &  2809.65 &  0.00 &  0.00 &  0.00\\
50 &  7200.00 &  7200.00 &  7200.00 &  2.57 &  0.81 &  3.86 &
50 &  5814.04 &  238.35 &  7200.00 &  0.35 &  0.01 &  1.26  \\
100 &  7200.00 &  7200.00 &  7200.00 &  6.96 &  4.36 &  8.86   &
100 &  7200.00 &  7200.00 &  7200.00 &  44.23 &  0.23 &  100.00   \\
\hline
Pre. + (LocP-F1) & 	&		&		&		 & 	&	  &  
Pre. + (LocP-F2) & 	&		&		&		 & 	&     \\ \hline
 $m$        &		&		&		&		 & 	&	  &  
 $m$        &		&		&		&		 & 	&	  \\ \hline
5 &  3.35 &  2.57 &  4.09 &  0.00 &  0.00 &  0.01     &
5 &  2.64 &  2.32 &  2.83 &  0.00 &  0.00 &  0.00    \\
10 &  9.61 &  8.84 &  10.67 &  0.00 &  0.00 &  0.01 &
10 &  11.17 &  8.09 &  14.41 &  0.00 &  0.00 &  0.01   \\
15 &  35.32 &  23.88 &  63.58 &  0.00 &  0.00 &  0.00  &
15 &  40.85 &  30.10 &  53.50 &  0.00 &  0.00 &  0.00\\
20 &  400.26 &  66.22 &  1525.33 &  0.00 &  0.00 &  0.01 &
20 &  87.53 &  51.91 &  186.05 &  0.00 &  0.00 &  0.01\\
30 &  3146.99 &  242.42 &  7200.00 &  0.38 &  0.00 &  1.84 &
30 &  908.80 &  175.44 &  3424.88 &  0.00 &  0.00 &  0.01\\
50 &  7200.00 &  7200.00 &  7200.00 &  2.32 &  0.12 &  4.60  &
50 &  1281.55 &  614.88 &  3083.29 &  0.01 &  0.00 &  0.01   \\
100 &  7200.00 &  7200.00 &  7200.00 &  6.07 &  3.39 &  9.27  &
100 &  7200.00 &  7200.00 &  7200.00 &  22.06 &  0.02 &  100.00  \\
\hline
\end{tabular}
}
\end{center}
\caption{Results obtained for the 4-points data set in \cite{Parlar1994}.}\label{tableL1}
\end{table}

\begin{table}[H]
\begin{center}
{\tiny
\begin{tabular}{|r|rrr|rrr|r|rrr|rrr|}
\hline
(LocP-F1) &  \multicolumn{3}{c|}{CPU(s)} & \multicolumn{3}{c|}{gap(\%)} & 
(LocP-F2) &  \multicolumn{3}{c|}{CPU(s)} & \multicolumn{3}{c|}{gap(\%)}\\ \hline
 $m$        &	aver	&	min	&	max	&	aver	 & min	&	max	&	  
 $m$        &	aver	&	min	&	max	&	aver	 & min	&	max\\ \hline
5 &  11.87 &  9.96 &  14.04 &  0.00 &  0.00 &  0.00  &
5 &  18.10 &  1.69 &  35.59 &  0.00 &  0.00 &  0.01\\
10 &  272.22 &  33.88 &  927.31 &  0.00 &  0.00 &  0.01 &
10 &  740.68 &  155.34 &  2638.86 &  0.00 &  0.00 &  0.00 \\
15 &  7200.00 &  7200.00 &  7200.00 &  3.32 &  0.06 &  6.66  &
15 &  4080.67 &  1286.86 &  7200.00 &  0.02 &  0.00 &  0.10\\
20 &  6921.33 &  5805.03 &  7200.00 &  1.96 &  0.01 &  5.67 &
20 &  3795.18 &  188.39 &  7200.00 &  4.45 &  0.00 &  19.35   \\
\hline
Pre. + (LocP-F1) & 	&		&		&		 & 	&	  &  
Pre. + (LocP-F2) & 	&		&		&		 & 	&     \\ \hline
 $m$        &		&		&		&		 & 	&	  &  
 $m$        &		&		&		&		 & 	&	  \\ \hline
5 &  10.69 &  6.48 &  15.66 &  0.00 &  0.00 &  0.00      &
5 &  15.68 &  7.11 &  31.05 &  0.00 &  0.00 &  0.01   \\
10 &  168.46 &  59.79 &  483.64 &  0.00 &  0.00 &  0.01   &
10 &  150.44 &  52.60 &  253.33 &  0.00 &  0.00 &  0.01  \\
15 &  7200.00 &  7200.00 &  7200.00 &  2.53 &  0.19 &  4.94  &
15 &  630.58 &  211.35 &  1087.40 &  0.00 &  0.00 &  0.01\\
20 &  6946.57 &  5934.32 &  7200.00 &  1.41 &  0.01 &  4.70  &
20 &  794.40 &  522.05 &  1319.05 &  0.00 &  0.00 &  0.01   \\
\hline
\end{tabular}
}
\end{center}
\caption{Results obtained for the 18-points data set in \cite{Parlar1994}.}\label{tableL2}
\end{table}

\begin{table}[H]
\begin{center}
{\tiny
\begin{tabular}{|r|rrr|rrr|r|rrr|rrr|}
\hline
(LocP-F1) &  \multicolumn{3}{c|}{CPU(s)} & \multicolumn{3}{c|}{gap(\%)} & 
(LocP-F2) &  \multicolumn{3}{c|}{CPU(s)} & \multicolumn{3}{c|}{gap(\%)}\\ \hline
 $m$        &	aver	&	min	&	max	&	aver	 & min	&	max	&	  
 $m$        &	aver	&	min	&	max	&	aver	 & min	&	max\\ \hline
5 &  105.39 &  8.83 &  480.18 &  0.00 &  0.00 &  0.01  &
5 &  1504.44 &  30.66 &  7200.00 &  20.00 &  0.00 &  100.00\\
10 &  2944.53 &  54.45 &  7200.00 &  0.29 &  0.00 &  1.04 &
10 &  2224.87 &  303.92 &  7200.00 &  20.00 &  0.00 &  100.00  \\
15 &  7200.00 &  7200.00 &  7200.00 &  3.61 &  2.02 &  6.80  &
15 &  5622.36 &  2353.16 &  7200.00 &  0.13 &  0.00 &  0.60\\
20 &  7200.00 &  7200.00 &  7200.00 &  3.26 &  1.68 &  6.83 &
20 &  6143.76 &  1881.24 &  7200.00 &  8.70 &  0.00 &  19.06\\
\hline
Pre. + (LocP-F1) & 	&		&		&		 & 	&	  &  
Pre. + (LocP-F2) & 	&		&		&		 & 	&     \\ \hline
 $m$        &		&		&		&		 & 	&	  &  
 $m$        &		&		&		&		 & 	&	  \\ \hline
5 &  28.87 &  8.98 &  79.53 &  0.00 &  0.00 &  0.01      &
5 &  30.89 &  12.41 &  53.97 &  0.00 &  0.00 &  0.00   \\
10 &  2960.38 &  87.97 &  7200.00 &  0.24 &  0.00 &  0.88   &
10 &  252.53 &  93.04 &  412.86 &  0.00 &  0.00 &  0.01  \\
15 &  7200.00 &  7200.00 &  7200.00 &  2.44 &  1.02 &  5.46  &
15 &  1176.88 &  673.57 &  1706.08 &  0.00 &  0.00 &  0.01\\
20 &  7200.00 &  7200.00 &  7200.00 &  1.69 &  0.35 &  2.75 &
20 &  4250.37 &  1137.07 &  7200.00 &  0.01 &  0.00 &  0.02\\
\hline
\end{tabular}
}
\end{center}
\caption{Results obtained for the 30-points data set in \cite{bsss}.}\label{tableL3}
\end{table}

\begin{table}[H]
\begin{center}
{\tiny
\begin{tabular}{|r|rrr|rrr|r|rrr|rrr|}
\hline
(LocP-F1) &  \multicolumn{3}{c|}{CPU(s)} & \multicolumn{3}{c|}{gap(\%)} & 
(LocP-F2) &  \multicolumn{3}{c|}{CPU(s)} & \multicolumn{3}{c|}{gap(\%)}\\ \hline
 $m$        &	aver	&	min	&	max	&	aver	 & min	&	max	&	  
 $m$        &	aver	&	min	&	max	&	aver	 & min	&	max\\ \hline
5 &  251.76 &  16.54 &  1180.78 &  0.00 &  0.00 &  0.01  &
5 &  1047.08 &  66.89 &  4423.39 &  0.00 &  0.00 &  0.01\\
10 &  3355.49 &  607.81 &  7200.00 &  0.77 &  0.01 &  2.82  &
10 &  1867.96 &  835.24 &  3380.85 &  0.00 &  0.00 &  0.00 \\
15 &  7200.00 &  7200.00 &  7200.00 &  5.11 &  2.46 &  7.82  &
15 &  7200.00 &  7200.00 &  7200.00 &  31.13 &  3.26 &  62.39\\
20 &  7200.00 &  7200.00 &  7200.00 &  6.53 &  3.58 &  10.83 &
20 &  7200.00 &  7200.00 &  7200.00 &  54.74 &  12.99 &  100.00 \\
\hline
Pre. + (LocP-F1) & 	&		&		&		 & 	&	  &  
Pre. + (LocP-F2) & 	&		&		&		 & 	&     \\ \hline
 $m$        &		&		&		&		 & 	&	  &  
 $m$        &		&		&		&		 & 	&	  \\ \hline
5 &  1464.15 &  23.72 &  7200.00 &  0.09 &  0.00 &  0.43      &
5 &  64.38 &  27.10 &  114.21 &  0.00 &  0.00 &  0.01   \\
10 &  3181.93 &  279.24 &  7200.00 &  0.45 &  0.00 &  1.48  &
10 &  464.81 &  239.18 &  652.83 &  0.00 &  0.00 &  0.01  \\
15 &  7200.00 &  7200.00 &  7200.00 &  3.19 &  1.45 &  5.08  &
15 &  3617.06 &  1589.11 &  7200.00 &  0.07 &  0.01 &  0.33\\
20 &  7200.00 &  7200.00 &  7200.00 &  3.20 &  1.64 &  4.30  &
20 &  6487.21 &  3545.33 &  7200.00 &  0.13 &  0.00 &  0.49\\
\hline
\end{tabular}
}
\end{center}
\caption{Results obtained for the 50-points data set in \cite{bsss}.}\label{tableL4}
\end{table}

\begin{table}[H]
\begin{center}
{\tiny
\begin{tabular}{|r|rrr|rrr|r|rrr|rrr|}
\hline
(LocP-F1) &  \multicolumn{3}{c|}{CPU(s)} & \multicolumn{3}{c|}{gap(\%)} & 
(LocP-F2) &  \multicolumn{3}{c|}{CPU(s)} & \multicolumn{3}{c|}{gap(\%)}\\ \hline
 $m$        &	aver	&	min	&	max	&	aver	 & min	&	max	&	  
 $m$        &	aver	&	min	&	max	&	aver	 & min	&	max\\ \hline
5 &  1453.43 &  14.12 &  7200.00 &  0.79 &  0.00 &  3.92  &
5 &  108.67 &  11.33 &  244.37 &  0.00 &  0.00 &  0.01\\
10 &  4084.76 &  1571.49 &  7200.00 &  1.44 &  0.00 &  6.19 &
10 &  1284.20 &  341.31 &  3208.91 &  0.00 &  0.00 &  0.01 \\
15 &  7200.00 &  7200.00 &  7200.00 &  5.24 &  2.51 &  11.87   &
15 &  6720.91 &  4752.57 &  7200.00 &  1.55 &  0.00 &  5.67\\
20 &  7200.00 &  7200.00 &  7200.00 &  9.71 &  5.68 &  14.64 &
20 &  7200.00 &  6974.36 &  7200.00 &  35.96 &  0.00 &  100.00\\
\hline
Pre. + (LocP-F1) & 	&		&		&		 & 	&	  &  
Pre. + (LocP-F2) & 	&		&		&		 & 	&     \\ \hline
 $m$        &		&		&		&		 & 	&	  &  
 $m$        &		&		&		&		 & 	&	  \\ \hline
5 &  1458.23 &  13.42 &  7200.00 &  0.73 &  0.00 &  3.08     &
5 &  70.77 &  20.86 &  133.53 &  0.00 &  0.00 &  0.01   \\
10 &  4680.45 &  371.35 &  7200.00 &  1.08 &  0.01 &  4.05   &
10 &  602.39 &  256.74 &  939.97 &  0.01 &  0.00 &  0.01  \\
15 &  7200.00 &  7200.00 &  7200.00 &  3.72 &  1.75 &  6.55  &
15 &  4311.12 &  1239.14 &  7200.00 &  0.18 &  0.00 &  0.72\\
20 &  7200.00 &  7200.00 &  7200.00 &  6.91 &  3.12 &  12.72 &
20 &  6315.88 &  2697.51 &  7200.00 &  14.07 &  0.00 &  69.17  \\
\hline
\end{tabular}
}
\end{center}
\caption{Results obtained for the 50-points data set in \cite{eilon-watson}.}\label{tableL5}
\end{table}

In the light of the results in Tables \ref{tableL1}-\ref{tableL5} we conclude that the preprocessing in Algorithm \ref{alg:preprocessLoc} is effective to lowering the gap for both formulation (LocP-F1) and formulation (LocP-F2). Especially significant is the effect of the preprocessing in the case of formulation (LocP-F2), where the high gap obtained by the formulation (100\% in some cases) in some instances is then drastically reduced to almost zero when the preprocessing is applied. However, the preprocessing is not able to induce such a high reduction in some instances, see Table \ref{tableL1} for $m=100$ or Table \ref{tableL5} for $m=20$ where the performance of the combined preprocessing and formulation (LocP-F2) is still poor. On the other hand, formulation (LocP-F1) shows a more regular performance in terms of gap although in general its results are worse. Formulation (LocP-F2) also provides better computation times than formulation (LocP-F1) (this behaviour can be explained by Proposition \ref{propRelaxationLoc}), so it may be advisable in general, in spite of its (rare) bad performance on a few instances.

Finally, we would like to point out a final fact on the comparison of our methods with that in \cite{refraction2017}. As commented above, we have generated the instances for Problem \eqref{LP} following a similar approach to the one in \cite{refraction2017} which  considers the data sets in \cite{eilon-watson,Parlar1994,bsss}. Indeed, the instances in \cite{refraction2017} are also instances of Problem \eqref{LP} for the particular case of only two regions. Therefore, it is an interesting question whether our formulations are also efficient when this particular case is considered. We have performed this computational experience and we have observed that the solutions found by our formulations coincide with the those reported in \cite{refraction2017} and that the computational times are also similar (of the order of one second or less).

\section{Conclusions}
\label{s:4}

In this paper the shortest path problem and the single-facility Weber problem are considered in a continuous framework subdivided in different polyhedra and where each polyhedron is endowed with a different $\ell_p$-norm. We derive two MISOCP formulations for each problem using the $\ell_p$-norms modeling procedure in \cite{BPE2014}. We prove that the continuous relaxation of one of the formulations provides tighter lower bounds than the other. Based on the tightness of the lower bounds provided by such a formulation, we develop a preprocessing algorithm which is used to improve the computational performance of the MISOCP formulations. We solve instances with up to 500 regions for the shortest path problem and also solve some classical instances adapted to the subdivision framework in the case of the Weber problem. As a by-product of the results in our work, we also relate the local optimality condition of the shortest path problem with Snell's law.

Besides the shortest path  and  Weber problems, other  problems also fit in the subdivision framework of this paper, as for instance the Travelling Salesman (TSP) and the Minimum Spanning Tree (MST) problems. The resulting problems are related with the TSP and MST with neighborhoods (see e.g. \cite{ArkHas1994} and \cite{BFP2017}, respectively). Indeed, we would obtain problems in which neighborhoods induce a subdivision into polyhedra and where the connection between neighborhoods is determined by the underlying graph of the subdivision and having different norms in each region. As we have shown, the formulations that we have presented are flexible since they can be easily adapted to handle different situations, which would also include the above TSP and MST problems. The adapted formulations would exploit the particular geometry of the subdivision framework for these problems. This subject is among our future lines of research and will be considered in a follow up paper.

\begin{acknowledgements}
The authors of this research acknowledge financial support by the Spanish Ministerio de Ciencia y Tecnologia, Agencia Estatal de Investigacion and Fondos Europeos de Desarrollo Regional (FEDER) via project PID2020-114594GB-C21, FEDER-US-1256951, Junta de Andalucía P18-FR-422, CEI-3-FQM331, FQM-331, and NetmeetData: Ayudas Fundación BBVA a equipos de investigación cient\'ifica 2019.
\end{acknowledgements}

\bibliographystyle{elsarticle-harv}

\begin{thebibliography}{99}

\bibitem{AMS2005}
Aleksandrov, L., Maheshwari, A., Sack, J. R. (2005). Determining approximate shortest paths on weighted polyhedral surfaces. Journal of the ACM (JACM), 52(1), 25-53.

\bibitem{ArkHas1994}
Arkin, E. M., Hassin, R. (1994). Approximation algorithms for the geometric covering salesman problem. Discrete Applied Mathematics, 55(3), 197-218.

\bibitem{BFP2017}
Blanco, V., Fern\'andez, E., Puerto, J. (2017). Minimum Spanning Trees with neighborhoods: Mathematical programming formulations and solution methods. European Journal of Operational Research, 262(3), 863-878.

\bibitem{BPE2014}
Blanco, V., Puerto, J., El-Haj Ben-Ali, S. (2014). Revisiting several problems and algorithms in continuous location with $\ell_\tau$ norms. Computational Optimization and Applications, 58(3), 563-595.

\bibitem{refraction2017}
Blanco, V., Puerto, J., Ponce, D. (2017). Continuous location under the effect of `refraction'. Mathematical Programming, 161(1-2), 33-72.

\bibitem{brimberg2003} 
Brimberg, J., Kakhki, H.T., Wesolowsky, G.O. (2003). \emph{Location Among Regions with Varying Norms}. Annals of Operations Research 122 (1-4), 87--102.

\bibitem{brimberg2005}
 Brimberg, J., Kakhki, H.T., Wesolowsky, G.O. (2005). \emph{ Locating a single facility in the plane in the presence of a bounded region and
different norms}. Journal of the Operational Research Society of Japan 48 (2), 135--47.

\bibitem{carrizosa-chia} Carrizosa, E. and Rodr\'iguez-Ch\'ia, A. (1997). \emph{Weber problems with alternative transportation systems}. European Journal of Operational Research 97 (1), 87--93.

\bibitem{CheeTomizuka1994}
Chee, W., Tomizuka, M. (1994, June). Lane change maneuver of automobiles for the intelligent vehicle and highway system (IVHS). In Proceedings of 1994 American Control Conference-ACC'94 (Vol. 3, pp. 3586-3587). IEEE.

\bibitem{CNVW2008}
Cheng, S. W., Na, H. S., Vigneron, A., Wang, Y. (2008). Approximate shortest paths in anisotropic regions. SIAM Journal on Computing, 38(3), 802-824.

\bibitem{DMNPY2008}
Daescu, O., Mitchell, J. S., Ntafos, S., Palmer, J. D., Yap, C. K. (2008). An experimental study of weighted k-link shortest path algorithms. In Algorithmic Foundation of Robotics VII (pp. 187-202). Springer, Berlin, Heidelberg.

\bibitem{unsolvabilityWRP}
De Carufel, J. L., Grimm, C., Maheshwari, A., Owen, M., Smid, M. (2014). A note on the unsolvability of the weighted region shortest path problem. Computational Geometry, 47(7), 724-727.

\bibitem{DH02}
Drezner, Z. and Hamacher, H.W. editors (2002). \emph{Facility Location: Applications and Theory}. Springer.

\bibitem{eilon-watson}
Eilon, S.,  Watson-Gandy, C., Christofides, N. (1971). Distribution management: mathematical modeling and practical analysis. Operational Research Quarterly, 20, 309.

\bibitem{fathaly} 
Fathali, J. and Zaferanieh, M. (2011). \emph{Location problems in regions with $\ell_p$ and block norms}. Iranian Journal of Operations Research 2 (2), 72--87.

\bibitem{FortSellares2012}
Fort, M., Sellar\`es, J. A. (2012). Approximating generalized distance functions on weighted triangulated surfaces with applications. Journal of Computational and Applied Mathematics, 236(14), 3461-3477.

\bibitem{FranVelGon2012}
Franco, L., Velasco, F., Gonzalez-Abril, L. (2012). Gate points in continuous location between regions with different $\ell_p$ norms. European journal of operational research, 218(3), 648-655.



\bibitem{GhandehariGolomb2000}
Ghandehari, M.,  Golomb, M. (2000). Minimum path problems in normed spaces: reflection and refraction. Journal of Optimization Theory and Applications, 105(1), 1-16.


\bibitem{pathrefinementWRP}
Gheibi, A., Maheshwari, A., Sack, J. R., Scheffer, C. (2018). Path refinement in weighted regions. Algorithmica, 80(12), 3766-3802.


\bibitem{Kreyszig1978} Kreyszig, E. (1978). Introductory functional analysis with applications (Vol. 1). New York: wiley.

\bibitem{LMS2001}
Lanthier, M., Maheshwari, A., Sack, J. R. (2001) Approximating Shortest Paths on Weighted Polyhedral Surfaces. Algorithmica 30, 527-562.

\bibitem{lyusternik1964}
Lyusternik, L. A. (1964). Shortest paths: variational problems. Macmillan Company, New York.

\bibitem{McCormick1976} McCormick, G. P. (1976). Computability of global solutions to factorable nonconvex programs: Part I-Convex underestimating problems. Mathematical programming, 10(1), 147-175.

\bibitem{Mangasarian1999} Mangasarian, O. L. (1999). Arbitrary-norm separating plane. Operations Research Letters, 24(1-2), 15-23.


\bibitem{matamitchell}
Mata, C. S., Mitchell, J. S. (1997). A new algorithm for computing shortest paths in weighted planar subdivisions. In Proceedings of the thirteenth annual symposium on Computational geometry, 264-273.

\bibitem{mitchell2000}
Mitchell, J. S. (2000). Geometric Shortest Paths and Network Optimization. Handbook of computational geometry, 334, 633-702.

\bibitem{mitchellWRP}
Mitchell, J. S., Papadimitriou, C. H. (1991). The weighted region problem: finding shortest paths through a weighted planar subdivision. Journal of the ACM (JACM), 38(1), 18-73.

\bibitem{Nelson1989}
Nelson, W. (1989, January). Continuous-curvature paths for autonomous vehicles. In 1989 IEEE International Conference on Robotics and Automation (pp. 1260-1261). IEEE Computer Society.

\bibitem{NP05}
Nickel, S., Puerto, J. (2005). Facility Location - A Unified Approach. Springer, Berlin.

\bibitem{NPC03}
Nickel, S., Puerto, J., Rodr\'iguez-Ch\'ia, A. M. (2003). An approach to location models involving sets as existing facilities. Mathematics of Operations Research, 28(4), 693-715.

\bibitem{NieKamMooOver2004}
Nieuwenhuisen, D., Kamphuis, A., Mooijekind, M.,  Overmars, M. H. (2004). Automatic construction of roadmaps for path planning in games. In International Conference on Computer Games: Artificial Intelligence, Design and Education (pp. 285-292).

\bibitem{Parlar1994} 
Parlar, M. (1994). Single facility location problem with region-dependent distance metrics. International journal of systems science, 25(3), 513--525.

\bibitem{plastria2019I} 
Plastria, F. (2019). Pasting gauges I: Shortest paths across a hyperplane. Discrete Applied Mathematics, 256, 105-137.

\bibitem{plastria2019II} 
Plastria, F. (2019). Pasting gauges II: Balls in pasted halfplanes. Discrete Applied Mathematics, 256, 138-156.

\bibitem{PuertoRCh2011} Puerto, J.,  Rodriguez-Chia, A.M. (2011) On the structure of the solution set for the single facility location problem with average distances. Mathematical Programming 128,373-401.

\bibitem{rockafellarCA}
Rockafellar, R. T. (1970). Convex analysis. Princeton University Press, Princeton.

\bibitem{Slawinski2000}
Slawinski, M. A., Slawinski, R. A., Brown, R. J., Parkin, J. M. (2000). A generalized form of Snell’s law in anisotropic media. Geophysics, 65(2), 632-637.

\bibitem{SR2006}
Sun, Z., Reif, J. H. (2006). On finding approximate optimal paths in weighted regions. Journal of Algorithms, 58(1), 1-32.

\bibitem{HandbookDCG2017}
Toth, C. D., O'Rourke, J., Goodman, J. E. (Eds.). (2017). Handbook of discrete and computational geometry. CRC press.

\bibitem{WardWendell1985}
Ward, J. E., Wendell, R. E. (1985). Using block norms for location modeling. Operations Research, 33(5), 1074-1090.

\bibitem{Warntz1957}
Warntz, W. (1957). Transportation, social physics, and the law of refraction. The Professional Geographer, 9(4), 2-7.

\bibitem{bsss} Zaferanieh, M., Taghizadeh Kakhki, H., Brimberg, J., Wesolowsky, G.O. (2008). A BSSS algorithm for the single facility location problem in two regions with different norms. European Journal of Operational Research, 190(1) 79--89.

\end{thebibliography}

\end{document}